\documentclass[11pt,reqno]{amsart}

\usepackage{srcltx}
\usepackage{amsmath,amssymb,mathtools,amsfonts,amscd,cases,color} \usepackage{hyperref}
\usepackage[capitalise]{cleveref}
\usepackage{vmargin} \setmarginsrb{1.5cm}{1cm}{1.5cm}{3cm}{1cm}{1cm}{2cm}{2cm} \usepackage{mathrsfs} \usepackage{graphicx} \usepackage{upref}

\hypersetup{linkcolor=blue, colorlinks=true,citecolor = red}

\newtheorem{theorem}{Theorem}[section] \newtheorem{proposition}[theorem]{Proposition}
  \newtheorem{remark}[theorem]{Remark} \newtheorem{defin}[theorem]{Definition}

\theoremstyle{definition}
 
\theoremstyle{remark}
\numberwithin{equation}{section}

 \newcommand{\mc}{\mathcal} \renewcommand{\leq}{\leqslant} \renewcommand{\geq}{\geqslant}

\begin{document}

\title[Self-propelled motion]{Self-propelled motion of a rigid body inside a density dependent incompressible fluid}

 \date{\today}

\author{\v S. Ne\u{c}asov\'a} \address{Institute of Mathematics, Czech Academy of Sciences, \v Zitn\' a 25, 11567 Praha 1, Czech Republic} \email{matus@math.cas.cz}

\author{M. Ramaswamy}  \address{Chennai Mathematical Institute, H1, SITCOT IT Park, Siruseri, 603103}  \email{mythily@cmi.ac.in}
 
\author{A. Roy}
\address{INRIA Nancy - Grand Est, 615, rue du
  Jardin Botanique. 54600 Villers-l\`es-Nancy, France}
 \email{royarnab244@gmail.com}

\author{A. Schl\"{o}merkemper} \address{Institute of Mathematics, University of W\"urzburg, Emil-Fischer-Str.~40, 97074 W\"urzburg, Germany}
\email {anja.schloemerkemper@mathematik.uni-wuerzburg.de}

\begin{abstract}
 This paper is devoted to the existence of a weak solution to a system describing a self-propelled motion of a rigid body in a viscous fluid in the whole $\mathbb{R}^3$. The fluid is modelled by the incompressible nonhomogeneous Navier-Stokes system with a nonnegative density. The motion of the rigid body is described by the  balance of linear and angular momentum. We consider the case where slip is allowed at the fluid-solid interface through Navier condition and prove the global existence of a weak solution. \end{abstract}

\maketitle

\bigskip



\section{Introduction}\label{sec_intro} Fluid-structure interaction (FSI) systems are systems which include a fluid and a solid component. The study of motion of small particles in fluids became one of the main focuses in research in the last 50 years. The presence of the particles has influence on the flow of the fluid and the fluid affects the motion of the particles. It implies that the problem of determining the flow characteristic is highly coupled. For such type
of everyday phenomena with a wide range of applications we refer to, e.g., \cite{FSIforBIO,ParticlesInFlow,G2} and references therein. Systems that
arise from modeling such phenomena are typically nonlinear systems of partial differential equations with a moving boundary or interface. 

We investigate a system where the rigid body moves in an incompressible Newtonian fluid which fills the whole space. The position of the rigid body at any given time moment is determined by two vectors describing the translation of the center of the mass and the rotation around the center of the mass, respectively. A system of six ordinary differential equations (Euler equations) describing the conservation of linear and angular momentum describes the dynamics of a rigid body.
 
 The fluid flow is governed by the three-dimensional incompressible Navier-Stokes  equations. The fluid domain depends on the position of the rigid body. The Navier-Stokes equations are coupled with a system of Euler ODE's via a dynamic and a kinematic coupling condition:
 The dynamic coupling condition is given by the balance of forces acting on the rigid body.
 There are several possibilities for the kinematic coupling condition.
 The no-slip condition, which postulates equality of the velocity of the fluid and of the structure on the boundary of the rigid body, is the most commonly used in the literature since it is the simplest to analyze and it is a physically reasonable condition. 
 However, in many situations, e.g., in close to contact dynamics (see, e.g., \cite{GVHil3,GVHil2}) or
 in the case of rough surfaces (see, e.g., \cite{Bucur2010,Masmoudi2010,JagerMikelic2001}),
  Navier's slip coupling condition is more appropriate since it allows for a discontinuity of the velocity in the tangential component  on the boundary of the rigid body.  In this article we therefore assume the Navier slip condition.
  
 Fluid-rigid body systems have been extensively studied in the last twenty years and some aspects of the well-posedness theory are now well
established. In the case of no self-propulsion, of homogeneous mass densities and constant viscosities, the existence of the unique local-in-time (or small data) solution is known in both two and three dimensions, and for both the slip \cite{SarkaSlipRigidStrong,WangFluidSolid} and the no-slip \cite{CumTak,GGH13,MaityTucsnak18,Takahashi03} coupling.
Further, it is known that a weak solution of Leray-Hopf type exists and is global in time or exists until the moment of contact between the boundary of the container and the rigid body for the slip \cite{ChemetovSarka,GVHil3,MR3165279} and the no-slip case \cite{ConcaRigid00,DE2,DE,GunzRigid00,Starovoitov02}.

One of the novelties of our article is that we will allow for self-propelled motion of the rigid body. Self-propulsion is a common means of  locomotion of macroscopic objects. Examples of such motions are performed by birds, fishes, rockets, submarines, etc. In microscopic world many organisms like ciliates, flagellates move by self-propulsion see, e.g., \cite{G2}. 
  
The case of no-slip boundary conditions for a moving rigid body with self-propulsion was considered in \cite{GAL1,Starovoitov} for different boundary conditions and homogeneous densities. The problem of the self-propelled motion as a control problem was tackled in \cite{Alouges,SanmartinTakahashiTucsnak,SigalottiVivalda} in the case of low Reynolds number (the fluid equations are the Stokes equations) and in \cite{Bressan,ChambrionMunnier} in the case of a potential fluid (high Reynolds number).
Let us also mention the work of Silvestre who investigated the slow motion of steady self-propelled motion and attainability \cite{MR1953783,Sil1}. 

Nonhomogeneous densities in viscous incompressible fluids were investigated in \cite{MR1422251, MR1062395}. To the best of our knowledge, nonhomogeneous densities were not treated in corresponding systems with structure interaction before.

In this article we consider the case of non-steady self-propelled motion in the case of a nonhomogeneous fluid (non-constant density) with Navier boundary conditions.   

To tackle our system, we first apply a global transformation such that the system is posed in a fixed but unbounded domain. In order to prove existence of a weak solution we first establish such an existence on a bounded domain $B_R = \{y\in\mathbb{R}^3 \, :  \, |y| < R\}$ by extending work of \cite{MR1062395} to the setting with a self-propelled rigid body. This is based on a Galerkin method. Due to the fluid-rigid structure interaction our test function depends here on the position of the body. Furthermore, we need to deal with the extra terms that are due to the global change of variables. We prove that all a-priori estimates are independent of $R$. We point out that the renormalized continuity equation enters in our definition of weak solutions to the fluid-structure problem. Up to our knowledge it is the first time that the renormalized continuity equation is considered in a nonhomogeneous fluid with structure. This way we obtain better regularity for the density. In the case without structure, the renormalized continuity equation was introduced by DiPerna, Lions \cite{MR1422251, DL}. 
The final step is then to show the existence of a weak solution in the unbounded domain by letting $R\to\infty$. Here we follow ideas by Planas and Sueur \cite[Theorem~1]{MR3165279}.

The novelties of our paper are as follows: (1) we include self-propelled motion of the rigid body,  (2) we allow for nonnegative and nonhomogeneous densities and (3) replace the no-slip condition at the interface of the solid and the fluid by the Navier slip condition. Furthermore, we consider (4) the case of a viscosity that depends on the mass density.

The outline of this article is as follows. The next section is devoted to an introduction of the mathematical problem, which is reformulated in a fixed domain. In Section~\ref{existence-bounded}, we prove the existence of weak solution in a bounded domain. Section~\ref{existence-unbounded} is devoted to the existence proof in an unbounded domain. In Section~\ref{sec:5} we mention the case of positive density and the case of a viscosity depending on the density and some further remarks and open problems.  In the appendix we present a derivation of the weak formulation of the problem. 
\smallskip

  \section{The mathematical model}

\smallskip

For $T>0$ given and for any time $t \in (0,T)$ let $\mathcal{S}(t) \subset \mathbb{R}^3$ denote a closed, bounded and simply connected  rigid body. We assume that the rest of the space, i.e., 
$\mathbb{R}^{3} \setminus \mathcal{S}(t)=\mathcal{F}(t)$ is filled with a viscous incompressible nonhomogeneous fluid. 
The initial domain of the rigid body is denoted by $\mc{S}_0$ and is assumed to have a smooth boundary. Correspondingly, $\mc{F}_0 =\mathbb{R}^{3} \setminus \mc{S}_0$ is  the initial fluid domain.

Our system of a rigid body moving in a fluid is in the first instance a moving domain problem in an inertial frame in which the velocity of the rigid body is described by
\begin{equation} \label{def-U_S} 
U_{\mc{S}}(x,t)= h'(t)+ R(t) \times (x-h(t)) \; \mbox{for} \quad (x,t) \in \mc{S}(t)\times (0,T). \end{equation} 
Here, $h'(t)$ is the linear velocity of the centre of mass $h$ and $R(t)$ is the angular velocity of the body. 
The solid domain at time $t$ in the same inertial frame is given by
\begin{equation*} \mc{S}(t)= \left\{h(t)+ Q(t)y \mid y\in \mc{S}_{0}\right\}, \end{equation*} 
where $Q(t) \in SO(3)$ is associated to the rotation of the rigid body.  Mathematically, it will turn out useful to define $U_{\mc{S}}$ for all $(x,t)\in \mathbb{R}^3 \times (0,T)$. 
Let $\varrho$, $U$ and $P$ be the density, velocity and pressure of the viscous
incompressible nonhomogeneous fluid, respectively, which satisfy 
\begin{align} \frac{\partial \varrho}{\partial t} + \operatorname{div}\,(\varrho U) &=0, \quad
\varrho(x) \geq 0  \quad \mbox{ in } \mathcal{F}(t) \times (0,T), \label{mass:fluid} \\ \frac{\partial }{\partial t}(\varrho U) + \operatorname{div}\left(\varrho U
\otimes U\right) + \nabla P &= \operatorname{div}(2\nu D(U)), \quad \mathrm{div}\,U=0 \quad \mbox{ in } \mathcal{F}(t) \times (0,T).
 \nonumber 
\end{align}
 Here we consider a self-propelled motion of the body $\mathcal{S}(t)$,  which we describe by a vectorial flux $W$ at $\partial \mc{S}(t)$. In this article we consider Navier type conditions at the fluid-structure interface; we prescribe the  normal and tangential parts of the flux by: 
\begin{align*} 
W\cdot N &= 0 \quad \mbox{ on } \partial \mc{S}(t) \times (0,T), \\
 \alpha (W \times N) &= (D(U)N)\times N + \alpha 
(U-U_{\mc{S}})\times N  \quad \mbox{ on } \partial \mc{S}(t)\times (0,T), 
\end{align*}
where $N$ is the unit outward normal to the boundary of $\mathcal{F}(t)$, i.e., directed towards $\mathcal{S}(t)$ and $\alpha$ is a given constant. We refer to Remark~\ref{rem:normal} for a discussion of the assumption that the normal component is zero. 

 For a given viscosity coefficient $\nu >0$, we set $$\Sigma (U,P) = -P\operatorname{I} + 2\nu D(U) \quad \mbox{ with } \quad D(U)= \left[\frac{1}{2}\left(\frac{\partial U_i}{\partial x_j} + \frac{\partial U_j}{\partial x_i}\right)\right]_{i,j=1,2,3}.$$
Let $\varrho_{\mc{S}}$ denote the density of the rigid body. Then $m = \int_{\mc{S}(t)} \varrho_{\mc{S}}(x,t) dx$ is its total mass, which is constant in time and does not change under the coordinate transformation below. Further, 
the moment of inertia $J(t)$ is defined by
$$J(t) =  \int\limits_{\mc{S}(t)} \rho_{\mc{S}}(x,t) 
 \Big(|x-h(t)|^2\mathrm{I} - (x-h(t))\otimes (x-h(t))\Big)\, dx. $$
The ordinary differential equations modeling the dynamics of the rigid body then read 
\begin{align*} mh''&= -\int\limits_{\partial \mc{S}(t)} \Sigma (U,P) N\, d\Gamma , 
\\ 
(JR)' &= -\int\limits_{\partial \mc{S}(t)} (x-h) \times \Sigma (U,P) N\, d\Gamma,
\end{align*}
We suppose the density and the velocity of the fluid to satisfy
\begin{equation*}
\varrho(x,t) \rightarrow 0,\quad U(x,t) \rightarrow 0 \mbox{ as }|x| \rightarrow \infty, 
\end{equation*}
as well as the initial conditions  \begin{equation}\label{initialcond} 
\varrho(x,0)=\rho_0(x) \geq 0,\,\varrho U(x,0)=q_0(x)=\rho_0(x)u_0(x),\, h(0)=0,\, h'(0)=\ell_{0},\, r(0)=r_{0}. 
\end{equation}
Following \cite{GAL1,MR3165279,Serre}, we apply a global change of variables that transforms the system in such a way that it is formulated in a frame which is attached to the rigid body. At time $t=0$, the two frames are assumed to coincide with $h(0) =0$. Hence the transformed system is posed on the fixed domain  $\mc{F}_{0} \times (0,T)$ via the following change of variables: 
\begin{align*} 
u(y,t) &=Q(t)^{\top}U(Q(t)y+h(t),t), \hspace{-3cm}& w(y,t) &= Q(t)^{\top} W(Q(t)y+h(t),t), \\
\rho(y,t)&=\varrho(Q(t)y+h(t),t), & p(y,t)&=P(Q(t)y+h(t),t).
\end{align*} 
Further, the moment of inertia transforms to 
$$ J_0 = Q(t)^\top J(t) Q(t).$$
By \eqref{def-U_S}
and the extension of $U_{\mc{S}}$ to the whole space, we have that  \color{black}
\begin{equation*} 
u_{\mc{S}}(y,t) = Q(t)^\top U_{\mc{S}}(Q(t)y + h(t),t)  =  \ell(t) + r(t) \times y, \quad y\in \mc{F}_0 \times (0,T),
\end{equation*}
where $\ell(t) = Q(t)^{\top}h'(t)$ is the transformed linear velocity and $r(t)=Q(t)^{\top}R(t)$ the transformed angular velocity of the rigid body. The normal and tangential parts of the self-propulsion flux in the new frame are given by: 
\begin{align} 
w\cdot n &= 0 \mbox{ on } \partial \mc{S}_0 \times (0,T), \label{self-propel normal}\\
 \alpha (w \times n) &= (D(u)n)\times n + \alpha 
(u-u_{\mc{S}})\times n  \mbox{ on } \partial \mc{S}_0 \times (0,T) \label{self-propel tangent}
\end{align}
with $n(y,t) = Q(t)^\top N(Q(t)y + h(t))$ for $(y,t) \in \partial \mc{S}_0 \times (0,T)$ denoting the inward pointing normal to $\partial \mc{S}_0$.
With all this at hand, equations \eqref{mass:fluid}--\eqref{initialcond} can be rewritten in the fixed domain as 
\begin{align} \frac{\partial\rho}{\partial t} + \operatorname{div}\,(\rho (u-u_{\mc{S}})) &=0, \quad \rho(y) \geq 0 \quad \mbox{ in } \mc{F}_0 \times (0,T), \label{chg of var mass:fluid} \\ \frac{\partial }{\partial t}(\rho u) + \operatorname{div}\left[u\otimes \rho(u-u_{\mc{S}})\right] + \rho r \times u + \nabla p &= \operatorname{div}(2\nu D(u)), \quad \mathrm{div}\,u=0 \quad \mbox{ in } \mc{F}_0 \times (0,T), \label{chg of var momentum:fluid} \end{align} 
\begin{align} u\cdot n &= u_{\mc{S}} \cdot n \mbox{ on } \partial \mc{S}_{0}\times (0,T), \label{chg of var boundary-1}\\ (D(u)n)\times n &= -\alpha (u-u_{\mc{S}} - w)\times n \mbox{ on } \partial \mc{S}_{0}\times (0,T), \label{chg of var boundary-2} \end{align} 
\begin{align} 
m\ell'&= -\int\limits_{\partial \mc{S}_{0}} \sigma (u,p) n\, d\Gamma + m\ell \times r, \label{chg of var linear momentum:body}\\ J_{0}r' &= -\int\limits_{\partial \mc{S}_{0}} y \times \sigma (u,p) n\, d\Gamma  + J_{0}r \times r, \label{chg of var angular momentum:body} \end{align} \begin{equation}\label{infinite behaviour} \rho(y,t) \rightarrow 0,\quad u(y,t) \rightarrow 0 \mbox{ as }|y| \rightarrow \infty, \end{equation} \begin{equation}\label{chg of var initial cond} \rho(y,0)=\rho_0(y) \geq 0,\,\rho u(y,0)=q_0(y)=\rho_0(y)u_0(y),\, h(0)=0,\, \ell(0)=\ell_{0},\, r(0)=r_{0}, \end{equation} where $$\sigma (u,p) = -p\operatorname{I} + 2\nu D(u)\mbox{ with  }D(u)= \left[\frac{1}{2}\left(\frac{\partial u_i}{\partial y_j} + \frac{\partial u_j}{\partial y_i}\right)\right]_{i,j=1,2,3}.$$


The system of partial and ordinary differential equations that we investigate in this work consists of the equations \eqref{chg of var mass:fluid}--\eqref{chg of var initial cond}.
  
\subsection{Notations and functional framework}
Before initiating our analysis, we collect here some basic notation that will be used throughout.  The linear space $\mathfrak{D}(\Omega)$ consists of
all infinitely differentiable functions that have compact support in $\Omega$.
The norm in a Lebesgue space $L^p$ (resp.\ Sobolev space $W^{k,p}$) is denoted by $\|\cdot\|_p$ (resp. $\|\cdot\|_{k,p}$). We denote by  $W^{-1,r}(\Omega)$ the dual space of $W_0^{1,r'}(\Omega)$, where $1/r + 1/r' =1$.
For convenience, we introduce
\begin{equation*}
H^1(\Omega)=W^{1,2}(\Omega),\ H^{-1}(\Omega)=W^{-1,2}(\Omega). 
\end{equation*}

We want to give an appropriate notion of weak solution to system \eqref{chg of var mass:fluid}--\eqref{chg of var initial cond}. In order to do so, 
following \cite{MR3165279}, we introduce the space of divergence free vector functions
\begin{equation*} \mc{H}=\left\{ \phi \in L^2(\mathbb{R}^3) \mid \mathrm{div}\, \phi = 0 \mbox{ in }\mathbb{R}^3 \mbox{ and }D(\phi)=0 \mbox{ in
}\mc{S}_{0} \right\}. \end{equation*}
For any $\phi \in \mc{H}$, there exist $\ell_{\phi} \in \mathbb{R}^3$ and $r_{\phi} \in \mathbb{R}^3$ such that 
\begin{equation} \label{eqn:phiS} \phi(y)= \ell_{\phi}
+ r_{\phi} \times y =: \phi_{\mc{S}}(y)  \quad \mbox{ for all }y \in \mc{S}_{0}.  \end{equation}
 Note that the tangential component of $\phi\in \mc{H}$ is allowed to jump at $\partial \mc{S}_0$, while the normal component has a continuous representative. We remark that, in the integrals on $\partial \mc{S}_0$ below, $\phi$ denotes the trace from the fluid side whereas $\phi_{\mc{S}}$ denotes the trace from the solid side. 
 
We consider the space $L^2(\mathbb{R}^3)$ with the following inner product: 
\begin{equation*} \left(\phi, \psi\right)_{\mc{H}} = \int\limits_{\mc{F}_{0}} \rho\phi \cdot \psi\,dy+ \int\limits_{\mc{S}_{0}} \rho_{\mc{S}_{0}} \phi \cdot \psi\, dy. \end{equation*} This inner product is equivalent to the usual inner product of $L^2(\mathbb{R}^3)$. Furthermore, when $\phi, \psi \in \mc{H}$, we have \begin{equation*}
\left(\phi, \psi\right)_{\mc{H}} = \int\limits_{\mc{F}_{0}} \rho\phi \cdot \psi\,dy + m\ell_{\phi}\cdot \ell_{\psi} + J_{0} r_{\phi}\cdot r_{\psi}. 
\end{equation*}
 The norm associated with the above inner product is denoted by $\|\cdot\|=\left(\cdot, \cdot\right)_{\mc{H}}$. Let us define the space 
 \begin{equation*} 
 \mc{V}=\left\{ \phi \in \mc{H} \mid \int\limits_{\mc{F}_{0}} |\nabla \phi(y)|^2\, dy < \infty \right\} \mbox{ with norm }\|\phi\|_{\mc{V}} = \|\phi\| + \|\nabla \phi\|_{L^2(\mc{F}_{0})}. 
 \end{equation*}

 Next we introduce some notations of time-dependent functions that we need later to describe the compactness properties. For $h>0$, the translated function of a function $f$ denoted as $\tau_h f$ is given by \begin{equation*} (\tau_h f)(t) = f(t+h) \mbox{ for any } t\in [0,T]. \end{equation*} Let $E$ be a Banach space.  For $1\leq q\leq \infty$, $0< s< 1$, Nikolskii spaces are defined by \begin{equation}\label{Nikolskii} N^{s,q}(0,T;E)=\left\{f\in L^q(0,T;E) \mid \sup_{h>0} h^{-s} \|\tau_hf-f\|_{L^q(0,T-h;E)} < \infty \right\}. \end{equation} 

\subsection{Energy inequality and definition of weak solution} Let $u$ be a smooth solution of \eqref{chg of var mass:fluid}--\eqref{chg of var initial cond} and $\phi \in C^{\infty}([0,T]; \mc{H})$  such that $\phi|_{\overline{\mc{F}_{0}}} \in C^{\infty}([0,T] \times \overline{\mc{F}_{0}})$. 
In the appendix we derive the weak form of our system. It reads 
\begin{align}\label{weakform-momentum}
 \lefteqn{\int\limits_{\mc{F}_{0}} \rho u(y,t)\cdot \phi(y,t)\,dy + m\ell(t)\cdot \ell_{\phi}(t)
 + J_{0}r(t)\cdot r_{\phi}(t)
  - \int\limits_{\mc{F}_{0}}q_0\cdot \phi(y,0)\,dy - m\ell(0)\cdot \ell_{\phi}(0) - J_{0}r(0)\cdot r_{\phi}(0) } \nonumber\\
&=  \int\limits_{0}^{t} \int\limits_{\mc{F}_{0}}\rho u \cdot \frac{\partial \phi}{\partial s}\,dy\,ds + \int\limits_{0}^{t}m\ell\cdot \ell'_{\phi}\,ds
+ \int\limits_{0}^{t}J_{0}r\cdot r'_{\phi}\,ds +\int\limits_{0}^{t}\int\limits_{\mc{F}_{0}} \left[u \otimes \rho(u-u_{\mc{S}})\right] : \nabla
\phi\,dy\,ds  \nonumber \\ &\quad  - \int\limits_{0}^{t}\int\limits_{\mc{F}_{0}} \operatorname{det} (\rho r, u, \phi)\,dy\,ds + \int\limits_{0}^{t}
\operatorname{det}(m\ell,r,\ell_{\phi})\,ds
 + \int\limits_{0}^{t} \operatorname{det}(J_{0}r,r,r_{\phi}) \,ds
- 2 \nu\int\limits_{0}^{t}\int\limits_{\mc{F}_{0}} D(u) : D(\phi)\,dy\,ds
 \nonumber\\
 &\quad - 2\nu\alpha\int\limits_{0}^{t}\int\limits_{\partial\mc{S}_{0}} (u-u_{\mc{S}}-w)\cdot (\phi-\phi_{\mc{S}})\,d\Gamma\,ds .
 \end{align}
This relation helps us to obtain an energy estimate for the system \eqref{chg of var mass:fluid}--\eqref{chg of var initial cond}. 
\begin{proposition} A  solution $(\rho,u,\ell,r)$ of the system \eqref{chg of var mass:fluid}--\eqref{chg of var initial cond} with $\rho \in C^\infty([0,T]\times \overline{\mc{F}_0})$, $u\in C^{\infty}([0,T]; \mc{H})$ such that $u|_{\overline{\mc{F}_{0}}} \in C^{\infty}([0,T] \times \overline{\mc{F}_{0}})$ and $\ell, r \in C^\infty([0,T])$ satisfy the following energy inequality: for almost every $t \in [0,T]$,
  \begin{align}\label{energy inequality}
 &\int\limits_{\mc{F}_{0}} \frac{1}{2}\rho |u|^2\,dy + \frac{m}{2}|\ell|^2
 + \frac{J_{0}}{2}|r|^2 + 2 \nu\int\limits_{0}^{t}\int\limits_{\mc{F}_{0}} |D(u)|^2\,dy\,ds +
 \nu\alpha\int\limits_{0}^{t}\int\limits_{\partial\mc{S}_{0}} |u-u_{\mc{S}}|^2\,d\Gamma\,ds \nonumber\\
&\leq \int\limits_{\mc{F}_{0}}\frac{1}{2}\rho_0 |u_0|^2\,dy + \frac{m}{2}|\ell_0|^2 + \frac{J_{0}}{2}|r_0|^2 +
\nu\alpha\int\limits_{0}^{t}\int\limits_{\partial\mc{S}_{0}}|w|^2\,d\Gamma\,ds.
 \end{align}
  \end{proposition}
  \begin{proof}
  To establish the energy inequality, we use the test function $\phi = u$ in \eqref{weakform-momentum}, to obtain
  \begin{align}\label{test equals solution}
 &\int\limits_{\mc{F}_{0}} \rho |u|^2\,dy + m|\ell|^2
 + J_{0}|r|^2
  - \int\limits_{\mc{F}_{0}}\rho_0 |u_0|^2\,dy - m|\ell_0|^2 - J_{0}|r_0|^2  \nonumber\\
&=  \int\limits_{0}^{t} \int\limits_{\mc{F}_{0}}\rho u \cdot \frac{\partial u}{\partial s}\,dy\,ds + \int\limits_{0}^{t}m\ell\cdot\ell'\,ds +
\int\limits_{0}^{t}J_{0}r\cdot r'\,ds +\int\limits_{0}^{t}\int\limits_{\mc{F}_{0}} \left[u \otimes \rho(u-u_{\mc{S}})\right] : \nabla u\,dy\,ds  \nonumber\\
&\quad - 2
\nu\int\limits_{0}^{t}\int\limits_{\mc{F}_{0}} |D(u)|^2\,dy\,ds - 2\nu\alpha\int\limits_{0}^{t}\int\limits_{\partial\mc{S}_{0}} |u-u_{\mc{S}}|^2\ d\Gamma\,ds
 + 2\nu\alpha\int\limits_{0}^{t}\int\limits_{\partial\mc{S}_{0}} w\cdot (u-u_{\mc{S}})\,d\Gamma\,ds.
 \end{align}
 In the above calculation, we have used that
$\operatorname{det} (\rho r, u, u)=
 \operatorname{det}(m\ell,r,\ell) =
  \operatorname{det}(J_{0}r,r,r) = 0$.
 Moreover,  
 \begin{align}
 \int\limits_{0}^{t}\int\limits_{\mc{F}_{0}} \left[u \otimes \rho(u-u_{\mc{S}})\right] : \nabla u\,dy\,ds =&
 \int\limits_{0}^{t}\int\limits_{\mc{F}_{0}} \left[ (\rho(u-u_{\mc{S}})\cdot \nabla)u\right] \cdot u\,dy\,ds \nonumber \\
 =& -\frac12\int\limits_{0}^{t}\int\limits_{\mc{F}_{0}} \operatorname{div} (\rho(u-u_{\mc{S}})) |u|^2\,dy\,ds  + \frac12 \int\limits_{0}^{t}\int\limits_{\partial \mc{S}_{0}} \rho(u-u_{\mc{S}})\cdot n
 |u|^2 \,d\Gamma\,ds \nonumber\\
 =&\frac12 \int\limits_{0}^{t}\int\limits_{\mc{F}_{0}} \frac{\partial \rho}{\partial s}|u|^2\,dy\,ds, \label{trilinear}
 \end{align} 
where the last equality follows from \eqref{chg of var mass:fluid} and \eqref{chg of var boundary-1}.
 By combining \eqref{test equals solution} and \eqref{trilinear}, we obtain
  \begin{align*}
 &\int\limits_{\mc{F}_{0}} \rho |u|^2\,dy + m|\ell|^2
 + J_{0}|r|^2
  - \int\limits_{\mc{F}_{0}}\rho_0 |u_0|^2\,dy - m|\ell_0|^2 - J_{0}|r_0|^2  \\
&=  \int\limits_{\mc{F}_{0}}\int\limits_{0}^{t}\frac{1}{2}\frac{d}{ds} (\rho |u|^2)\,ds\,dy + \int\limits_{0}^{t}\frac{m}{2}\frac{d}{ds}|\ell|^2\,ds +
\int\limits_{0}^{t}\frac{J_{0}}{2}\frac{d}{ds}|r|^2\,ds - 2 \nu\int\limits_{0}^{t}\int\limits_{\mc{F}_{0}} |D(u)|^2\,dy\,ds\\
 & \quad- 2\nu\alpha\int\limits_{0}^{t}\int\limits_{\partial\mc{S}_{0}} |u-u_{\mc{S}}|^2\,d\Gamma\,ds
 + 2\nu\alpha\int\limits_{0}^{t}\int\limits_{\partial\mc{S}_{0}} w\cdot (u-u_{\mc{S}})\,d\Gamma\,ds.
 \end{align*}
 Thus,
 \begin{align*}
 & \int\limits_{\mc{F}_{0}} \frac{1}{2}\rho |u|^2\,dy + \frac{m}{2}|\ell|^2
 + \frac{J_{0}}{2}|r|^2 + 2 \nu\int\limits_{0}^{t}\int\limits_{\mc{F}_{0}} |D(u)|^2\,dy\,ds +
 2\nu\alpha\int\limits_{0}^{t}\int\limits_{\partial\mc{S}_{0}} |u-u_{\mc{S}}|^2\,d\Gamma\,ds  \\
&= \int\limits_{\mc{F}_{0}}\frac{1}{2}\rho_0 |u_0|^2\,dy + \frac{m}{2}|\ell_0|^2 + \frac{J_{0}}{2}|r_0|^2 +
2\nu\alpha\int\limits_{0}^{t}\int\limits_{\partial\mc{S}_{0}} w\cdot (u-u_{\mc{S}})\,d\Gamma\,ds\\ &\leq \int\limits_{\mc{F}_{0}}\frac{1}{2}\rho_0
|u_0|^2\,dy + \frac{m}{2}|\ell_0|^2 + \frac{J_{0}}{2}|r_0|^2 +
\nu\alpha\int\limits_{0}^{t}\int\limits_{\partial\mc{S}_{0}}|u-u_{\mc{S}}|^2\,d\Gamma\,ds +
\nu\alpha\int\limits_{0}^{t}\int\limits_{\partial\mc{S}_{0}}|w|^2\,d\Gamma\,ds
 \end{align*}
and \eqref{energy inequality} follows.
\end{proof}
 The relation \eqref{weakform-momentum} motivates us to define weak solutions in the following way.
 \begin{defin}
Let $T> 0$. A pair $(\rho,u)$ is a weak solution to system \eqref{chg of var mass:fluid}--\eqref{chg of var initial cond} if the following conditions hold true:
\begin{itemize} \item $\rho \geq 0, \quad \rho \in L^{\infty}((0,T)\times \mathbb{R}^3)$. 
\item $u \in L^2(0,T; \mc{V}),\quad \rho|u|^2 \in L^{\infty}(0,T; L^1(\mathbb{R}^3))$.  
\item The equation of continuity \eqref{chg of var mass:fluid} is satisfied in the weak sense, i.e., 
\begin{equation*} \int\limits_{\mathbb{R}^3} \Big(\rho(y,T)\phi(y,T) -\rho_0\phi(y,0)\Big) \, dy  = \int\limits_{0}^{T}\int\limits_{\mathbb{R}^3} \left[\rho \frac{\partial \phi}{\partial s} + \rho (u-u_{\mc{S}}): \nabla \phi\right]\,dy\,ds, \end{equation*} 
for any test function $\phi \in \mathfrak{D}([0,T)\times \mathbb{R}^3)$. 
Also, a renormalized continuity equation holds in a weak sense, i.e., \begin{equation*} 
\int\limits_{\mathbb{R}^3} b(\rho)\phi \, dy\Big|_{0}^{T}  = \int\limits_{0}^{T}\int\limits_{\mathbb{R}^3} b(\rho)\left[\frac{\partial \phi}{\partial s} + (u-u_{\mc{S}}): \nabla \phi\right]\,dy\,ds, 
\end{equation*} 
for any test function $\phi \in \mathfrak{D}([0,T)\times \mathbb{R}^3)$ and for all $b \in C^1(\mathbb{R})$. 
\item Balance of linear momentum holds in a weak sense, i.e., for all  $\phi \in C^{\infty}([0,T]; \mc{H})$  such that $\phi|_{\overline{\mc{F}_{0}}} \in C^{\infty}([0,T] \times \overline{\mc{F}_{0}})$  and for all $t \in [0,T]$, the relation
    \eqref{weakform-momentum} holds. 
\end{itemize}
  \end{defin}

  \begin{remark}
  \smallskip
  \begin{itemize}
  \item There is no a priori reason that the momentum $\rho u$ is continuous in time. We only have that, for any $t_0 \in [0,T]$,
  the function
     \begin{equation*}
     t \mapsto \int_{\mathbb{R}^3} (\rho u)(t,y)\cdot \phi(y) dy 
     \end{equation*}
is continuous in a certain neighbourhood of $t_0$, provided
$\phi = \phi(y) \in  \mathfrak{D}(\mathbb{R}^3)$ and $D(\phi)=0$
on a neighbourhood of $\mc S(t_0)$. 
\item The introduction of the renormalized continuity equation in the definition of weak solutions to our system yields that  $ \rho \in C([0,T]; L^q_{\operatorname{loc}}(\mathbb{R}^3))$ for all $q\in [1,\infty)$, see below for details. 
\end{itemize}
  \end{remark}
 In the following theorem we state the main result of our paper regarding the global existence of a weak solution to system \eqref{chg of var mass:fluid}--\eqref{chg of var initial cond}.
  \begin{theorem}\label{main theorem}
 Let $\mc{S}_0$ be a bounded,  closed, simply connected set with smooth boundary and $\mc{F}_0= \mathbb{R}^3\setminus \mc{S}_0$. Assume that $w$ satisfies \eqref{self-propel normal}--\eqref{self-propel tangent} and that there exist
 constants $c_1, c_2 > 0$ such that 
 \begin{align*}
&\rho_0 \geq 0,\, \rho_0 |_{\mc{S}_0} \in [c_1, c_2],\, \rho_0 \in L^{\infty}(\mathbb{R}^3),\,  u_{0} \in \mc{H},\\ & q_0=0\mbox{ a.e.\  on }\{\rho_0 = 0\},\ \frac{q_0^2}{\rho_0} \in L^1(\mc{F}_0).
 \end{align*}
 Then, for any $T>0$ there exists a weak solution $(\rho,u)$ to system \eqref{chg of var mass:fluid}--\eqref{chg of var initial
 cond}.
  Moreover, we have 
 \begin{equation*}
 \inf_{\mathbb{R}^3}\rho_0 \leq \rho \leq \sup_{\mathbb{R}^3}\rho_0, \quad \rho \in C([0,T]; L^q_{\operatorname{loc}}(\mathbb{R}^3))\,\forall\ q\in [1,\infty),\quad p \in W^{-1,\infty}(0,T; L^2(\mc{F}_0)),
 \end{equation*}
and for a.e.\
  $t \in [0,T]$, the energy inequality \eqref{energy inequality} holds.
  \end{theorem}
  In order to prove the main result, as a first step, we will prove the existence of a weak solution in a bounded domain $B_R$, where $B_R=\{y \in
  \mathbb{R}^3 \mid |y| < R\}$, see Section~\ref{existence-bounded}. Thereafter, we take $R$ to infinity to establish the existence of a weak solution for the whole space $\mathbb{R}^3$ in Section~\ref{existence-unbounded}.

\section{Existence of a weak solution in a bounded domain}\label{existence-bounded} 
In this section we consider the system \eqref{chg of var mass:fluid}--\eqref{chg of var initial
cond} in a smooth bounded domain $\Omega\subset \mathbb{R}^3$, along with the complementary boundary condition \begin{equation}\label{boundary-bdd domain} u(y,t)= 0, \quad y\in
\partial \Omega. \end{equation}
We assume $\mc{S}_0 \subset \Omega$ and set $\mc{F}_0 = \Omega \setminus \mc{S}_0$.  Let us define a weak solution to system \eqref{chg of var mass:fluid}--\eqref{chg of var initial cond} in any bounded
domain $\Omega$. To do that, we introduce two spaces $\mc{H}_{\Omega}$, $\mc{V}_{\Omega}$  for the bounded domain
$\Omega$, which are analogous to $\mc{H}$, $\mc{V}$. We set 
\begin{equation*} \mc{H}_{\Omega}=\left\{ \phi \in L^2({\Omega}) \mid \mathrm{div}\, \phi = 0 \mbox{ in }{\Omega} \mbox{ and
}D(\phi)=0 \mbox{ in }\mc{S}_{0} \right\}, \end{equation*} \begin{equation*} \mc{V}_{\Omega}=\left\{ \phi \in \mc{H}_{\Omega} \mid \phi=0 \mbox{ on
}\partial\Omega, \int\limits_{\mc{F}_{0}} |\nabla \phi(y)|^2\, dy < \infty \right\}. \end{equation*}
 \begin{defin}
Let $T> 0$ and let $\Omega\subset \mathbb{R}^3$ be a smooth bounded domain. A pair $(\rho,u)$ is a weak solution to system 
\eqref{chg of var mass:fluid}--\eqref{chg of var initial cond} with \eqref{boundary-bdd domain} if the following conditions hold true:
\begin{itemize} \item $\rho \geq 0, \quad \rho\in L^{\infty}((0,T) \times \Omega)$. 
\item $u \in L^2(0,T; \mc{V}_{\Omega})$, $\rho|u|^2 \in L^{\infty}(0,T; L^1(\Omega))$. 
\item The equation of continuity \eqref{chg of var mass:fluid} is satisfied in the weak sense, i.e., 
\begin{equation*} \int\limits_{\Omega} \Big(\rho(y,T)\phi(y,T) -\rho_0\phi(y,0)\Big) \, dy  = \int\limits_{0}^{T}\int\limits_{\Omega} \left[\rho \frac{\partial \phi}{\partial s} + \rho (u-u_{\mc{S}}): \nabla \phi\right]\,dy\,ds, \end{equation*} 
for any test function $\phi \in \mathfrak{D}([0,T)\times \Omega)$. 
Also, a renormalized continuity equation holds in a weak sense, i.e., \begin{equation*} 
\int\limits_{\Omega} b(\rho)\phi \, dy\Big|_{0}^{T}  = \int\limits_{0}^{T}\int\limits_{\Omega} b(\rho)\left[\frac{\partial \phi}{\partial s} + (u-u_{\mc{S}}): \nabla \phi\right]\,dy\,ds, 
\end{equation*} 
for any test function $\phi \in \mathfrak{D}([0,T)\times \Omega)$ and for all $b \in C^1(\mathbb{R})$. 
\item Balance of linear momentum holds in a weak sense, i.e., for all  $\phi \in C^{\infty}([0,T];
\mc{H}_{\Omega})$  such that $\phi|_{\overline{\mc{F}_{0}}} \in C^{\infty}([0,T] \times \overline{\mc{F}_{0}})$  and for all $t \in [0,T]$, the relation
    \eqref{weakform-momentum} holds. 
\end{itemize}
  \end{defin}
  Next we assert the existence of a weak solution in a bounded domain.
  \begin{theorem}\label{existence:bounded domain}
 Let $R$ be sufficiently large and $\Omega=B_R$. Set $\mc{F}_0= \Omega\setminus \mc{S}_0$, where $\mc{S}_0\subset B_R$ is a bounded, closed, simply connected set with smooth boundary. Assume that $w$ satisfies \eqref{self-propel normal}--\eqref{self-propel tangent} and that there exist constants $c_1, c_2 > 0$ such that
 \begin{align*}
&\rho_0 \geq 0,\, \rho_0 |_{\mc{S}_0} \in [c_1, c_2],\, \rho_0 \in L^{\infty}(\Omega),\,  u_{0} \in \mc{H}_{\Omega},\\ & q_0=0\mbox{ a.e.\  on }\{\rho_0 = 0\},\ \frac{q_0^2}{\rho_0} \in L^1(\mc{F}_0).
 \end{align*}
 Then for any time $T>0$ there exists a weak solution $(\rho_R,u_R)$ to system \eqref{chg of var mass:fluid}--\eqref{chg of var initial
 cond} satisfying \eqref{boundary-bdd domain} on the time interval $(0,T)$. Moreover, we have
 \begin{align*}
 &\inf_{\Omega}\rho_0 \leq \rho_{R} \leq \sup_{\Omega}\rho_0, \quad \rho_{R} \in C([0,T]; W^{-1,\infty}(\Omega)) \cap C([0,T]; L^q(\Omega))\, \quad \forall\ q\in [1,\infty),\\
  &p_{R}\in W^{-1,\infty}(0,T; L^2(\mc{F}_0)),
 \end{align*}
  and for a.e.\ $t \in [0,T]$, the energy inequality \eqref{energy inequality} holds for $(\rho_R,u_R)$.
 \end{theorem}
 \begin{proof}
 The proof is divided into several steps. At first we construct an approximate solution and then establish the existence of a weak solution to system
 \eqref{chg of var mass:fluid}--\eqref{chg of var initial cond} as the limit of this approximation.
 \vspace{.2cm}
 
\noindent \underline{Step 1: Construction of $N-$th level approximate solution}

 Since the set
\begin{equation*} \mc{Y}=\left\{\xi \in C^1(\overline{\Omega})\cap \mc{V}_{\Omega}\mid \operatorname{div}\xi =0\mbox{ in }\Omega\mbox{ and
}D(\xi)=0\mbox{ in }\mc{S}_0\right\} \end{equation*} is dense in $\mc{V}_{\Omega}$, we can choose an orthonormal  basis $z_i \in \mc{Y}$, for all $i\geq 1$ of the
Hilbert space $\mc{V}_{\Omega}$. Define the $N$-dimensional space $X_N$ by \begin{equation*} X_N = \mbox{span} \{z_i\}_{1\leq i \leq N}.
\end{equation*} 
If $u_N(t),z_j \in X_N$, then there exist $\ell_{N}(t),\ell_{z_j} \in \mathbb{R}^3$ and $r_{N}(t),r_{z_j} \in \mathbb{R}^3$ such that
\begin{align*}
 u_N(y,t)&= \ell_{N}(t) + r_{N}(t) \times y, \mbox{ for all }y \in \mc{S}_{0},\\ 
 z_j(y)&= \ell_{z_j} + r_{z_j} \times y, \mbox{ for all }y \in \mc{S}_{0}.
 \end{align*} 
Let us define the following quantities:
 \begin{equation*}
u_{\mc{S},N}= \ell_{N}(t) + r_{N}(t) \times y, \quad        z_{\mc{S},j}= \ell_{z_j} + r_{z_j} \times y, \mbox{ for all }(y,t) \in \mc{F}_{0}\times (0,T).
 \end{equation*}
 We are looking for $u_N, \rho_N$, the solution of the approximate problem, such that for some $T_N >0$, 
\begin{equation}\label{approx1} 
 u_N \in C^1([0,T_N],X_N),\quad \rho_N \in C^1([0,T_N], C^1(\overline{\Omega})) 
\end{equation} 
 satisfy for all $z_j \in X_N $:
\begin{align}\label{approx:integral equation}
 &\int\limits_{\mc{F}_{0}} \rho_N \frac{\partial u_N}{\partial s} \cdot z_j\,dy + m\ell'_N \cdot \ell_{z_j} + J_{0} r'_N \cdot r_{z_j} \nonumber\\
 &= -\int\limits_{\mc{F}_{0}} \left[(\rho_N(u_N-u_{\mc{S},N})\cdot \nabla)u_N\right]\cdot z_j\,dy
  - \int\limits_{\mc{F}_{0}} \operatorname{det} (\rho_N r_N, u_N, z_j)\,dy
 +  \operatorname{det}(m\ell_N,r_N,\ell_{z_j})  + \operatorname{det}(J_{0}r_N,r_N,r_{z_j}) \nonumber\\ &\quad -  2 \nu\int\limits_{\mc{F}_{0}} D(u_N) : D(z_j)\,dy
 - 2\nu\alpha\int\limits_{\partial\mc{S}_{0}} (u_N-u_{\mc{S},N})\cdot (z_j-z_{\mc{S},j})\,d\Gamma + 2\nu\alpha\int\limits_{\partial\mc{S}_{0}} w\cdot(z_j-z_{\mc{S},j})\,d\Gamma,
 \end{align}
 \begin{align}
 & \quad \frac{\partial\rho_N}{\partial t} + \operatorname{div}\,(\rho_N (u_N-u_{\mc{S},N})) =0,  \quad \rho_N(y) > 0 \mbox{ in } \mc{F}_0 \times (0,T_N), \label{approx2}\\
 &\quad u_N(0)=u_{0N},\quad \rho_N(0)=\rho_{0N} \label{approx3}.
 \end{align}
 Here $u_{0N}$ and $\rho_{0N}$ are functions satisfying
 \begin{align} 
  &u_{0N} \in X_N, \hspace{-5cm}& u_{0N} &\rightarrow u_0 \mbox{ in }L^2(\mc{F}_0) \mbox{ as } N\to \infty \label{u1}\\
  &\rho_{0N} \in C^1(\overline{\Omega}), \hspace{-5cm}& \rho_{0N} &\rightharpoonup \rho_0\mbox{ in }L^{\infty}(\Omega)\mbox{ weak-$*$ as }N\rightarrow
     \infty,\label{rho1}\\
 & \frac{1}{N} + \inf \rho_0 \leq \rho_{0N} \leq \frac{1}{N} + \sup \rho_0 \hspace{-5cm}&&\label{rho2}.
 \end{align}
 Further we have 
 \begin{equation*}
 \ell_{0N} \rightarrow \ell_0 \mbox{ in }\mathbb{R}^3, \quad
     \omega_{0N} \rightarrow \omega_0 \mbox{ in }\mathbb{R}^3 \mbox{ as } N\to\infty.
  \end{equation*}     
We study the local existence of $u_N$, $\rho_N$ by similar arguments as in \cite[Theorem 9]{MR1062395}, or in \cite[Chapter VI, Theorem VI.2.1]{MR2986590}. We define the Banach space $$E_N=C([0,T_N],X_N).$$ 
We will construct a map $$\mc{N}: E_N \rightarrow E_N$$ which allows to find a fixed point that is a solution to the approximate problem. We  do so in three steps.

Firstly, let $v_N \in E_N$ be given. We consider 
 \begin{equation}\label{transport:fixedpoint}
 \frac{\partial\rho_N}{\partial t} + \operatorname{div}\,(\rho_N (v_N-v_{\mc{S},N})) =0,
 \quad \rho_N(0)=\rho_{0N} \mbox{ in } \in \mathcal{F}_0 \times (0,T_N).
 \end{equation}
  We define the trajectory $Y^{N}=Y^{N}_{x,t}(s)$ of a particle located at $x$ at time $t$ as
\begin{equation*} 
\frac{dY^{N}}{ds}(s)= v_N(Y^N(s),s), \mbox { for all } s \geq 0, Y^N(t)=x.
\end{equation*}
Since $v_N \in C([0,T_N],X_N)$, we have $Y^N \in C^1([0,T_N];C^1(\bar {\Omega}))$. For fixed $v_N$, \eqref{transport:fixedpoint} has a unique solution $$\rho _N(x,t)=\rho_{N0}(Y^{N}_{x,t}(t)).$$ such that $\rho _{N}(0)=\rho _{N0}$.

Secondly, after the construction of $\rho_N$, we are looking for $u_N \in E_N$ satisfying
\begin{align}\label{momentum:fixedpoint}
 &\int\limits_{\mc{F}_{0}} \rho_N \frac{\partial u_N}{\partial s} \cdot z_j\,dy + m\ell'_N \cdot \ell_{z_j} + J_{0} r'_N \cdot r_{z_j} \nonumber \\ &=
 -\int\limits_{\mc{F}_{0}} \left[(\rho_N(v_N-v_{\mc{S},N})\cdot \nabla)u_N\right]\cdot z_j\,dy  - \int\limits_{\mc{F}_{0}} \operatorname{det} (\rho_N r_N, u_N, z_j)\,dy
 +  \operatorname{det}(m\ell_N,r_N,\ell_{z_j})  + \operatorname{det}(J_{0}r_N,r_N,r_{z_j}) \nonumber \\
  &\quad -  2 \nu\int\limits_{\mc{F}_{0}} D(u_N) : D(z_j)\,dy  - 2\nu\alpha\int\limits_{\partial\mc{S}_{0}} (u_N-u_{\mc{S},N})\cdot (z_j-z_{\mc{S},j})\,d\Gamma + 2\nu\alpha\int\limits_{\partial\mc{S}_{0}} w\cdot(z_j-z_{\mc{S},j})\,d\Gamma
 \end{align}
 for all $z_j \in X_N$. Now we seek the solution $u_N$ of \eqref{momentum:fixedpoint} in the form $$u_N(x,t)=\sum_{i=1}^{N}\alpha_{jN}(t)z_i(x), \mbox{ for any }j=1,2,...,N.$$ 
Let us introduce the matrices and vectors
\begin{equation*}
\mc{M}_N=((z_i,z_j)_{\mc{H}})_{1\leq i,j \leq N},\quad \alpha_N=(\alpha_{jN})_{1\leq j\leq N},\quad \mc{A}_N=(a(z_i,z_j))_{1\leq i,j \leq N},\quad \mc{C}=(c(w,z_j))_{1\leq j\leq N},
\end{equation*}
where 
\begin{align*}
a(u_N,v_N) &= -  2 \nu\int\limits_{\mc{F}_{0}} D(u_N) : D(v_N)\,dy
 - 2\nu\alpha\int\limits_{\partial\mc{S}_{0}} (u_N-u_{\mc{S},N})\cdot (v_N-v_{N,\mc{S}})\,d\Gamma,\\
 c(w,v)&= 2\nu\alpha\int\limits_{\partial\mc{S}_{0}} w\cdot(v-v_{\mc{S}})\,d\Gamma.
\end{align*}
Furthermore, for any $u,v \in \mathbb{R}^N$, $\mc{B}_N(u,v)=(\mc{B}_{Nj}(u,v))_{1\leq j \leq N}$, where
\begin{align*}
\mc{B}_{Nj}(u,v)=&-\sum\limits_{1\leq k\leq N}v_k\int\limits_{\mc{F}_{0}} \left[(\rho_N(v_N-v_{\mc{S},N})\cdot \nabla)z_k\right]\cdot z_j\,dy \\ &+\sum\limits_{1\leq i,k\leq N}u_iv_k\Big(- \int\limits_{\mc{F}_{0}} \operatorname{det} (\rho_N r_{z_i}, z_k, z_j)\,dy 
 +  \operatorname{det}(m\ell_{z_i},r_{z_k},\ell_{z_j})  + \operatorname{det}(J_{0}r_{z_i},r_{z_k},r_{z_j})\Big).
\end{align*}
 Thus, equation \eqref{momentum:fixedpoint} can be viewed as
 \begin{equation}\label{ODEform}
 \mc{M}_N\frac{d\alpha_{N}}{dt}=\mc{A}_N\alpha_{N} + \mc{B}_{N} + \mc{C},
 \end{equation}
 where $\alpha_N\in \mathbb{R}^N$ is the unknown vector, $\mc{A}_N, \mc{M}_N$ are $N\times N$ matrices and $\mc{B}_N,\mc{C}\in \mathbb{R}^N$ are vectors defined as above. As $(\cdot,\cdot)_{\mc{H}}$ is a scalar product on $L^2(\mathbb{R}^3)$ and $z_i$ is an orthonormal basis, we have that $\mc{M}_N$ is invertible. Hence, by Cauchy-Lipschitz theory, equation \eqref{ODEform} with the  corresponding initial condition has a unique solution in some interval $[0,T_N]$. As a consequence, there exists a unique $u_N \in E_N$ which satisfies \eqref{momentum:fixedpoint}.
 
Thirdly, given $v_N \in E_N$, we have obtained a unique couple $(\rho_N,u_N)$ that solves \eqref{transport:fixedpoint}--\eqref{momentum:fixedpoint}. Now we define the map $\mc{N}: E_N \rightarrow E_N$ by $$\mc{N}(v_N)=u_N.$$
Following the steps of \cite[Chapter VI, Theorem VI.2.1]{MR2986590}, we obtain that the map $\mc{N}$ has a fixed point in a suitable subset of $E_N$. This fixed point (denoted by $u_N$) along with $\rho_N$ is the solution of the nonlinear approximate problem \eqref{approx:integral equation}--\eqref{approx3}.
\vspace{.2cm}

\noindent\underline{Step 2: Global existence  }
 \vspace{.2cm}\\
Since the velocity is smooth, we obtain from \eqref{approx2} and \eqref{rho2} that 
 \begin{equation}\label{est:approx-rho}
 \frac{1}{N} + \inf \rho_0 \leq \rho_{N} \leq \frac{1}{N} + \sup \rho_0.
 \end{equation}
As in \eqref{trilinear} we have 
 \begin{equation}\label{trilinear-approx}
  \int\limits_{\mc{F}_{0}} \left[ (\rho(u_N-u_{\mc{S},N})\cdot \nabla)u_N\right] \cdot u_N\,dy= \frac{1}{2}\int\limits_{\mc{F}_{0}} \frac{\partial \rho_N}{\partial s}|u_N|^2\,dy.
 \end{equation}
 We take $z_j=u_N(t) \in X_N$ as the test function in \eqref{approx:integral equation}, and by using \eqref{trilinear-approx}, we obtain that
\begin{align*}
 &\frac{1}{2}\frac{d}{ds}\left(\int\limits_{\mc{F}_{0}} \rho_N |u_N|^2\,dy + m|\ell_N|^2 + J_{0} |r_N|^2 \right) +  2\nu\int\limits_{\mc{F}_{0}}|D(u_N)|^2\,dy + 2\nu\alpha\int\limits_{\partial\mc{S}_{0}} |u_N-u_{\mc{S},N}|^2\,d\Gamma \\
 &=2\nu\alpha\int\limits_{\partial\mc{S}_{0}} w\cdot (u_N-u_{\mc{S},N})\,d\Gamma.
 \end{align*}
 Integrating the above equation from $0$ to $t$ and using H\"{o}lder's inequality, we get
 \begin{align}\label{energy-approx}
 & \int\limits_{\mc{F}_{0}} \frac{1}{2}\rho_N |u_N|^2\,dy + \frac{m}{2}|\ell_N|^2
+ \frac{J_{0}}{2}|r_N|^2 + 2 \nu\int\limits_{0}^{t}\int\limits_{\mc{F}_{0}} |D(u_N)|^2\,dy\,ds +
\nu\alpha\int\limits_{0}^{t}\int\limits_{\partial\mc{S}_{0}} |u_N-u_{\mc{S},N}|^2\,d\Gamma\,ds \nonumber \\ & \leq \int\limits_{\mc{F}_{0}}\frac{1}{2}\rho_{0,N}
|u_{0,N}|^2\,dy + \frac{m}{2}|\ell_{0,N}|^2 + \frac{J_{0}}{2}|r_{0,N}|^2 +
\nu\alpha\int\limits_{0}^{t}\int\limits_{\partial\mc{S}_{0}}|w|^2\,d\Gamma\,ds.
 \end{align}
 We already know that the system \eqref{approx1}--\eqref{approx3} has a maximal solution in some time interval $[0,T_N]$ with $T_N > 0$. If $T_N < T$, then
 $\|u_N\|_{\mc{H}}$ must tend to $\infty$ as $t\rightarrow T$. But the energy inequality for approximate solution \eqref{energy-approx} suggests that
 $\|u_N\|_{\mc{H}}$ is finite as $t\rightarrow T$. Thus, $T_N=T$ and with the help of \eqref{est:approx-rho} and \eqref{energy-approx}, we have
 \begin{itemize}
 \item $u_N \mbox{ is bounded in } L^2(0,T; \mc{V}_{\Omega})$,
 \item $\rho_N$ is bounded in $L^{\infty}((0,T)\times \Omega)$,
 \item $\sqrt{\rho_N}u_N$ is bounded in $L^{\infty}(0,T; L^2(\Omega))$,
  \item $\ell_N$, $r_N$ are bounded in $L^{\infty}(0,T)$.
 \end{itemize}
 Furthermore, by using \cite[Proposition 7]{MR1062395} and \cite[Proposition 8]{MR1062395}, we obtain
 \begin{itemize}
 \item $\dfrac{\partial \rho_N}{\partial t}$ is bounded in $L^{2}(0,T; W^{-1,6}(\Omega)) \cap L^{\infty}(0,T; H^{-1}(\Omega))$,
 \item $\rho_Nu_N$ is bounded in $L^2(0,T; L^6(\Omega))\cap L^{\infty}(0,T; L^2(\Omega))$,
 \item $u_N \otimes \rho_N u_N$ is bounded in $L^{4/3}(0,T; L^2(\Omega))$, 
 \item $\rho_Nu_N$ is bounded in $N^{1/4,2}(0,T; W^{-1,3/2}(\Omega))$ (Nikolskii spaces are introduced in \eqref{Nikolskii}). 
 \end{itemize}
 \vspace{.2cm}
 
 \noindent \underline{Step 3: Compactness argument and convergence properties}
 \vspace{.2cm}\\
 We show that the approximate solution $(\rho_N,u_N)$ constructed in step~1 has a limit in suitable spaces and its limit is a solution to system \eqref{chg of var mass:fluid}--\eqref{chg of var initial cond} with \eqref{boundary-bdd domain}.
 
 Let $X\subset E \subset Y$ be Banach spaces and the imbedding $X \hookrightarrow E$ be compact.

 We recall Aubin-Lions-Simon Lemma (\cite[Lemma 4(ii)]{MR1062395}, \cite[Corollary 4]{MR916688}):  If $F$ is bounded in $L^{\infty}(0,T;X)$ and 
$\dfrac{\partial
 F}{\partial t}$ is bounded in $L^r(0,T;Y)$ for some $r>1$, then $F$ is relatively compact in $C([0,T];E)$. If we consider $X=L^{\infty}(\Omega)$, $E= W^{-1,\infty}(\Omega)$, $Y=W^{-1,6}(\Omega)$ and $r=2$, the properties of $\rho_N$ enlisted above imply
 \begin{equation*}
 \{\rho_N\}_{N\in \mathbb{N}} \mbox{ is relatively compact in }C([0,T];W^{-1,\infty}(\Omega)).
 \end{equation*}
 Moreover, \cite[Lemma 4(iv)]{MR1062395} states that if $X\subset E \subset Y$, $X \hookrightarrow E$ is compact and $F \in L^q(0,T;X)\cap N^{s,q}(0,T;Y)$, then $F$ is relatively compact in $L^q(0,T;E)$ for $s>0$,
 $1\leq q \leq \infty$. Hence, if we take $X=L^2(\Omega)$, $E=H^{-1}(\Omega)$, $Y=W^{-1,3/2}(\Omega)$, $q=2$ and $s=1/4$, the properties of
 $\rho_N$ enlisted above imply
 \begin{equation*}
 \{\rho_Nu_N\}_{N\in \mathbb{N}} \mbox{ is relatively compact in }L^2(0,T; H^{-1}(\Omega)).
 \end{equation*}
 Therefore, we can extract a subsequence of $\{\rho_N, u_N\}_{N\in \mathbb{N}}$,  relabelled the same, such that
 \begin{itemize}
 \item  $u_N \rightharpoonup u$ in $L^2(0,T; \mc{V}_{\Omega})$ weakly,
 \item $\rho_N \rightharpoonup \rho$ in $L^{\infty}((0,T)\times \Omega)$ weak-$*$ and in $C([0,T];W^{-1,\infty}(\Omega))$ strongly,
 \item  $\rho_Nu_N \rightharpoonup g$  in $L^{\infty}(0,T; L^2(\Omega))$ weak-$*$ and in $L^2(0,T; H^{-1}(\Omega))$ strongly,
 \item $u_N \otimes \rho_N(u_N-u_{\mc{S},N}) \rightharpoonup k$ in $L^{4/3}(0,T; L^2(\Omega))$ weakly.
 \end{itemize}
Let $\Omega \subset \mathbb{R}^3$ be a bounded set and let $1\leq r < 3$, $1\leq s \leq \infty$ with
$\dfrac{1}{r} + \dfrac{1}{s} \leq 1$. According to \cite[Lemma 3(iii)]{MR1062395}, the imbedding of the  product of two Sobolev spaces \begin{equation*} W^{1,r}(\Omega)\times
W^{-1,s}(\Omega)\rightarrow W^{-1,t}(\Omega), \quad \frac{1}{t}=\frac{1}{r}-\frac{1}{3}+\frac{1}{s}, \end{equation*}
 is continuous.

 Thus, if we consider $r=2$, $s=\infty$, then $t=6$ and the product is continuous from $H^1(\Omega) \times W^{-1,\infty}(\Omega)$ into
 $W^{-1,6}(\Omega)$. The strong convergence of $\rho_N \rightarrow \rho$ in $C([0,T];W^{-1,\infty}(\Omega))$ and the weak convergence of $u_N
 \rightharpoonup u$ in $L^2(0,T; \mc{V}_{\Omega})$ imply that
 \begin{equation*}
 \rho_Nu_N \rightharpoonup \rho u\mbox{ in }L^2(0,T; W^{-1,6}(\Omega))\mbox{ weakly}.
 \end{equation*}
 Hence applying $ \rho_Nu_N \rightharpoonup g \in L^{\infty}(0,T; L^2(\Omega))$ weak-$*$ and in $L^2(0,T; H^{-1}(\Omega))$ strongly,
  we get  $$g=\rho u.$$
 On the other hand, if we consider $r=2$, $s=2$, then $t=\dfrac{3}{2}$ and the product is continuous from $H^1(\Omega) \times H^{-1}(\Omega)$
 into $W^{-1,3/2}(\Omega)$. The weak convergence of $u_N \rightharpoonup u$ in $L^2(0,T; \mc{V}_{\Omega})$ and the strong convergence of $\rho_N u_N$ in $L^2(0,T;H^{-1}(\Omega))$ imply that
 \begin{equation*}
 u_N \otimes \rho_N(u_N-u_{\mc{S},N}) \rightharpoonup u \otimes (g-\rho u_{\mc{S}})\mbox{ in }L^1(0,T; W^{-1,3/2}(\Omega))\mbox{ weakly}.
 \end{equation*}
 Hence, $$k= u \otimes \rho(u-u_{\mc{S}}).$$
 We recall the interpolation inequality: let $p_1,p_2,q_1,q_2$ be four numbers in $[1,\infty]$. If $f\in L^{p_1}(0,T;L^{q_1}(\Omega)) \cap L^{p_2}(0,T;L^{q_2}(\Omega))$, then for all $\theta\in (0,1)$, the function $f\in L^{p}(0,T;L^{q}(\Omega))$
 with the estimate 
 \begin{equation*}
 \|f\|_{L^{p}(0,T;L^{q}(\Omega))} \leq \|f\|^{\theta}_{L^{p_1}(0,T;L^{q_1}(\Omega))}\|f\|^{1-\theta}_{L^{p_2}(0,T;L^{q_2}(\Omega))},
 \end{equation*}
 where 
 \begin{equation*}
 \dfrac{1}{p}=\dfrac{\theta}{p_1}+\dfrac{1-\theta}{p_2},\quad \dfrac{1}{q}=\dfrac{\theta}{q_1}+\dfrac{1-\theta}{q_2}.
 \end{equation*}
 Hence, we obtain
 \begin{align*}
 \|\sqrt{\rho_N}u_N\|_{L^{\frac{8}{3}}(0,T;L^4(\Omega))} \leq & \|\sqrt{\rho_N}u_N\|^{1/4}_{L^{\infty}(0,T;L^2(\Omega))}\|\sqrt{\rho_N}u_N\|^{3/4}_{L^{2}(0,T;L^6(\Omega))} \\\leq & C \|\sqrt{\rho_N}u_N\|^{1/4}_{L^{\infty}(0,T;L^2(\Omega))}\|u_N\|^{1/4}_{L^{2}(0,T;H^1(\Omega))}.
 \end{align*}
 Thus, the sequence $u_N \otimes \rho_N(u_N-u_{\mc{S},N})$ is bounded in the space $L^{\frac{4}{3}}(0,T;L^2(\Omega))$. We recall a result of weak convergence (\cite[Chapter II, Proposition II.2.10]{MR2986590}): let $E,F$ be Banach spaces with $F$ reflexive and $\{x_n\}$ be a sequence in $E\cap F$ such that there exists $x\in E$ satisfying $\{x_n\}$ is bounded in $F$, $\{x_n\}$ converges weakly to $x$ in $E$. Then, $\{x_n\}$ converges weakly to $x$ in $F$. 
 
 In our case, if we consider $E=L^1(0,T; W^{-1,3/2}(\Omega))$ and $F=L^{\frac{4}{3}}(0,T;L^2(\Omega))$, then we have 
 \begin{equation*}
 u_N \otimes \rho_N(u_N-u_{\mc{S},N}) \rightharpoonup u \otimes \rho(u- u_{\mc{S}})\mbox{ in }L^{\frac{4}{3}}(0,T;L^2(\Omega))\mbox{ weakly}.
 \end{equation*} 
 
 \vspace{.2cm}
 \noindent\underline{Step 4: $N\rightarrow \infty$}
 \vspace{.2cm}\\
 As $\rho_N \rightarrow \rho$ in $\mathfrak{D}'((0,T)\times \Omega)$ and $\rho_Nu_N \rightarrow \rho u$ in $\mathfrak{D}'((0,T)\times \Omega)$ as
 $N\rightarrow \infty$, we can pass to the limit in $\mathfrak{D}'((0,T)\times \Omega)$ in the equation \eqref{approx2} and obtain
 \begin{equation}\label{limit:rho}
\frac{\partial\rho}{\partial t} + \operatorname{div}\,(\rho (u-u_{\mc{S}}))=0,\mbox{ in } \mathfrak{D}'((0,T)\times \Omega) \mbox{ and }\rho\geq 0.
\end{equation}
 The convergence properties of $\rho_N \rightarrow \rho$ and $u_N \rightarrow u$ as $N\rightarrow \infty$, (discussed in Step 2) yield
 \begin{equation*} \rho\in L^{\infty}((0,T) \times \Omega),\quad u \in L^2(0,T; \mc{V}_{\Omega}).
\end{equation*} 
Similarly, we can also check that the limiting variables $(\rho,u)$ satisfy the renormalized continuity equation. Precisely, using the regularization lemma (see \cite[Lemma 2.3, page 43]{MR1422251}), we can deduce that as in \cite[Theorem 2.4]{MR1422251}: if $u \in L^2(0,T; \mc{V}_{\Omega})$, $\rho\in L^{\infty}((0,T) \times \Omega)$, $\operatorname{div} u=0$ a.e.\ and $\rho$ satisfies \eqref{limit:rho}, then for any $b \in C^1(\mathbb{R})$, $b(\rho)$ also satisfies
\eqref{limit:rho}. Thus, the continuity equation is also satisfied in the renormalized sense. This regularization lemma and renormalization arguments also help us to establish (see \cite[Chapter VI, Theorem VI.1.9]{MR2986590}, \cite[Theorem 2.4]{MR1422251}) the fact that:
\begin{equation*}
\rho_N \rightarrow \rho \mbox{ in }C([0,T];L^q(\Omega)), \mbox{ for any }q\in[1,\infty).
\end{equation*} 

In order to prove
the existence of a weak solution $(\rho,u)$, it remains to verify the relation \eqref{weakform-momentum} for all $t \in [0,T]$ and $\phi \in
C^{\infty}([0,T]; \mc{H}_{\Omega})$ such that  $\phi|_{\overline{\mc{F}}_{0}} \in C^{\infty}([0,T]\times \overline{\mc{F}}_{0}))$.  To this end, we pass to the limit in equation \eqref{approx:integral equation}. We know from \cite[Chapter V, Lemma V.1.2]{MR2986590} that the set
\begin{equation*}
\Big\{\xi_k(y)\psi_k(t) \mid \xi_k \in \mc{Y},
\psi_k \in C^{\infty}([0,T])\Big\} \mbox{ is dense in }C^{\infty}([0,T]; \mc{H}_{\Omega}).
\end{equation*}
Hence, it suffices to  verify the relation
\eqref{weakform-momentum} for all $\phi(y,t) =\xi(y)\psi(t)$, where $\xi \in \mc{Y}$, $\psi \in C^{\infty}([0,T])$.

We consider equation \eqref{approx:integral equation} with $z_j$ replaced with $\phi(y,t) =\xi(y)\psi(t)$ and use the relation
\begin{equation*} \int\limits_{\mc{F}_{0}} \left[
(\rho_N(u_N-u_{\mc{S},N})\cdot \nabla)u_N\right] \cdot \phi\,dy = \int\limits_{\mc{F}_0} \frac{\partial \rho_N}{\partial t}u_N\cdot \phi\,dy -
\int\limits_{\mc{F}_0} \left[u_N \otimes \rho_N(u_N-u_{\mc{S},N})\right] : \nabla \phi \,dy, \end{equation*}
 which follows similarly as \eqref{trilinear}. We then  obtain  \begin{align*}
 & \int\limits_{\mc{F}_{0}}\frac{d}{ds} (\rho_N u_N(y,s))\cdot \phi(y,s)\,dy + m\ell'_N(s)\cdot \ell_{\phi}(s)
 + J_{0}r'_N(s)\cdot r_{\phi}(s) \\ 
 & =  \int\limits_{\mc{F}_{0}} \left[u_N \otimes \rho_N(u_N-u_{\mc{S},N})\right] : \nabla \phi\,dy 
  - \int\limits_{\mc{F}_{0}} \operatorname{det} (\rho_N r_N, u_N, \phi)\,dy
+ \operatorname{det}(m\ell_N,r_N,\ell_{\phi}) + \operatorname{det}(J_{0}r_N,r_N,r_{\phi}) \\ &\quad - 2 \nu\int\limits_{\mc{F}_{0}} D(u_N) : D(\phi)\,dy
 - 2\nu\alpha\int\limits_{\partial\mc{S}_{0}} (u_N-u_{\mc{S},N})\cdot (\phi-\phi_{\mc{S}})\ d\Gamma
 + 2\nu\alpha\int\limits_{\partial\mc{S}_{0}} w\cdot (\phi-\phi_{\mc{S}})\,d\Gamma.
 \end{align*}
The main problematic term is the first one on the left hand side. To deal with this term, we integrate from $0$ to $t$ and apply the product rule so that the derivative is on the test function
$\phi$. In this term as well as in the other terms, we can then pass to the limit as $N \rightarrow \infty$  by using the
convergence results obtained in Step~3: 
\begin{align*}
u_N &\rightharpoonup u\mbox{ in }L^2(0,T; \mc{V}_{\Omega})\mbox{ weakly},\\
 \rho_Nu_N &\rightharpoonup \rho u\mbox{ in }L^2(0,T; W^{-1,6}(\mc{F}_0))\mbox{ weakly},\\
 u_N \otimes \rho_N(u_N-u_{\mc{S},N}) &\rightharpoonup u \otimes \rho(u- u_{\mc{S}})\mbox{ in }L^{\frac{4}{3}}(0,T;L^2(\mc{F}_0))\mbox{ weakly}.
 \end{align*}
Finally, we obtain
 \begin{align*}
 &\int\limits_0^t\frac{d}{ds}\left(\int\limits_{\mc{F}_{0}} \rho u(y,s)\cdot \phi(y,s)\,dy + m\ell(s)\cdot \ell_{\phi}(s)
 + J_{0}r(s)\cdot r_{\phi}(s)\right)\,ds \\
& = \int\limits_{0}^{t}\left( \int\limits_{\mc{F}_{0}} \rho u \cdot \frac{\partial \phi}{\partial s}\,dy + m\ell\cdot \ell'_{\phi} + J_{0}r\cdot
r'_{\phi}\right)\,ds  +\int\limits_{0}^{t}\int\limits_{\mc{F}_{0}} \left[u \otimes \rho(u-u_{\mc{S}})\right] : \nabla \phi\,dy\,ds \\
&\quad -
\int\limits_{0}^{t}\int\limits_{\mc{F}_{0}} \operatorname{det} (\rho r, u, \phi)\,dy\,ds + \int\limits_{0}^{t}
\operatorname{det}(m\ell,r,\ell_{\phi})\,ds
 + \int\limits_{0}^{t} \operatorname{det}(J_{0}r,r,r_{\phi}) \,ds
 \\
& \quad - 2 \nu\int\limits_{0}^{t}\int\limits_{\mc{F}_{0}} D(u) : D(\phi)\,dy\,ds
 - 2\nu\alpha\int\limits_{0}^{t}\int\limits_{\partial\mc{S}_{0}} (u-u_{\mc{S}})\cdot (\phi-\phi_{\mc{S}})\ d\Gamma\,ds
 + 2\nu\alpha\int\limits_{0}^{t}\int\limits_{\partial\mc{S}_{0}} w\cdot (\phi-\phi_{\mc{S}})\,d\Gamma\,ds.
 \end{align*}
 This equality clearly yields \eqref{weakform-momentum}.
 \vspace{.5cm}\\
 \underline{Step 5: Existence of pressure}
 \vspace{.2cm}\\
 Let us introduce
 \begin{equation*}
 \mathscr{V}=\left\{\mathfrak{D}(\mc{F}_0) \mid \mathrm{div}\, \phi = 0 \mbox{ in }{\Omega} \mbox{ and }D(\phi)=0 \mbox{ in }\mc{S}_{0} \right\}.
 \end{equation*}
 We choose an appropriate test function $\phi \in C^{\infty}([0,T];\mathscr{V})$ in the relation \eqref{weakform-momentum} (so that the boundary
 terms vanish after integrate by parts) to obtain
 \begin{equation*}
 \left\langle\frac{\partial }{\partial t}(\rho u) + \operatorname{div}\left[u\otimes \rho(u-u_{\mc{S}})\right] + \rho r \times u -
 \operatorname{div}(2\nu D(u)), \phi\right\rangle=0, \quad \forall \,\phi \in C^{\infty}([0,T];\mathscr{V}).
 \end{equation*}
 Moreover, according to \cite[Proposition 7(ii)]{MR1062395},
 \begin{equation*}
 \frac{\partial }{\partial t}(\rho u) + \operatorname{div}\left[u\otimes \rho(u-u_{\mc{S}})\right] + \rho r \times u - \operatorname{div}(2\nu D(u))
 \in W^{-1,\infty}(0,T; H^{-1}(\mc{F}_0)).
 \end{equation*}
 Thus, by \cite[Lemma 2]{MR1062395}, there exists $p\in W^{-1,\infty}(0,T; L^{2}(\mc{F}_0))$ such that
 \begin{equation*}
 \frac{\partial }{\partial t}(\rho u) + \operatorname{div}\left[u\otimes \rho(u-u_{\mc{S}})\right] + \rho r \times u - \operatorname{div}(2\nu D(u))= \nabla p, \mbox{ in } \mathcal{F}_{0} \times (0,T).
 \end{equation*}
\end{proof} 

\section{Proof of \cref{main theorem}}\label{existence-unbounded} 

We prove the existence of a weak solution in the unbounded domain by using \cref{existence:bounded domain} (existence of solution in a ball $B_R$) and letting $R\rightarrow \infty$.

\begin{proof} Let $R_0$ be such that $\mc{S}_0 \subset B(0,{R_0}/{2})$. Choose $R>R_0$ and consider a smooth function $\chi_R : \mathbb{R}^3\rightarrow
\mathbb{R}^3$ as in \cite[Theorem 1]{MR3165279} defined by $$ \chi_R(y)=\begin{cases} y \mbox{ for }y\in B(0,R),\\ \dfrac{R}{|y|}y \mbox{ for }y\in
\mathbb{R}^3 \setminus B(0,R) \end{cases} $$ 
and set 
\begin{equation*} u_{\mc{S},R}(y,t) = \ell(t) + r(t) \times
\chi_R(y). \end{equation*} Observe that for any $r\in \mathbb{R}^3$ and $\phi\in \mc{V}$: \begin{equation*} (r\times \chi_R)\cdot \nabla \phi
\rightarrow (r\times y)\cdot \nabla \phi \mbox{ in }L^2(\mathbb{R}^3)\mbox{ as }R\rightarrow \infty, \end{equation*} and \begin{equation*}
u_{\mc{S},R}(y,t)\rightarrow \ell(t)+r(t)\times y=u_{\mc{S}}(y,t)\mbox{ as }R\rightarrow \infty. \end{equation*} \cref{existence:bounded domain}
suggests that the bounds obtained in \eqref{energy inequality} for $(\rho_R,u_R)$ are uniform with respect to $R$. Thus the maximum principle and
energy inequality \eqref{energy inequality} help us to conclude \begin{align*} \|\rho_R\|_{L^{\infty}((0,T)\times \mathbb{R}^3)} &\leq C,\\
\|\sqrt{\rho_R}u_R\|_{L^{\infty}(0,T;L^2(\mathbb{R}^3))} &\leq C,\\ \|u_R\|_{L^2(0,T;\mc{V})} &\leq C. 
\end{align*} 
This implies that up to an
extraction of a subsequence we have $\rho_R \rightharpoonup \rho$ in $L^{\infty}$ weak-$*$ in $C([0,T];L^{q}_{\operatorname{loc}} (\mathbb{R}^3))$ for any $q\in [1,\infty)$ and $u_R \rightharpoonup u$ in $L^2(0,T;\mc{V})$ weak. Following
\cite[Section 1]{MR1744863}, we can show that $\rho_R \rightarrow \rho$ in $C([0,T];W^{-1,2}(\mathbb{R}^3))$, $\rho_Ru_R\rightarrow \rho u$ in $\mathfrak{D}'((0,T)\times \mathbb{R}^3)$ and $\rho$ satisfies
  \begin{equation*}
\frac{\partial\rho}{\partial t} + \operatorname{div}\,(\rho (u-u_{\mc{S}}))=0\mbox{ in } \mathfrak{D}'((0,T)\times \mathbb{R}^3). \end{equation*} 
By \cite[Lemma 1]{MR1744863}, $\rho$ also satisfies the renormalized continuity equation. On the other hand, regularization lemma \cite[Lemma 2.3]{MR1422251} and renormalization arguments help us to establish (see \cite[Chapter VI, Theorem VI.1.9]{MR2986590}, \cite[Theorem 2.4]{MR1422251}) that
\begin{equation*}
\|\rho_R - \rho\|_{L^{\infty}([0,T];L^q_{\operatorname{loc}} (\mathbb{R}^3))} \rightarrow 0 \mbox{ as } R\to \infty \mbox{ for any }q\in[1,\infty).
\end{equation*}
Thus, we have $\rho_R \rightarrow \rho$ strongly in $C([0,T];L^{q}_{\operatorname{loc}} (\mathbb{R}^3))$.
 It remains to establish
identity \eqref{weakform-momentum}. We already know from \cref{existence:bounded domain} that $(\rho_R,u_R)$ satisfy \eqref{weakform-momentum}, i.e., 
\begin{align}\label{weakform-momentum-bounded}
 & \int\limits_{\mc{F}_{0}} \rho_R u_R(y,t)\cdot \phi(y,t)\,dy + m\ell(t)\cdot \ell_{\phi}(t)
 + J_{0}r(t)\cdot r_{\phi}(t)
  - \int\limits_{\mc{F}_{0}}q_0\cdot \phi(y,0)\,dy - m\ell(0)\cdot \ell_{\phi}(0) - J_{0}r(0)\cdot r_{\phi}(0)  \nonumber \\
& =  \int\limits_{0}^{t} \int\limits_{\mc{F}_{0}} \rho_R u_R \cdot \frac{\partial \phi}{\partial s}\,dy \,ds+ \int\limits_{0}^{t}m\ell\cdot
\ell'_{\phi}\,ds + \int\limits_{0}^{t}J_{0}r\cdot r'_{\phi}\,ds +\int\limits_{0}^{t}\int\limits_{\mc{F}_{0}} \left[u_R \otimes
\rho_R(u_R-u_{\mc{S},R})\right] : \nabla \phi\,dy\,ds \nonumber\\
& \quad - \int\limits_{0}^{t}\int\limits_{\mc{F}_{0}} \operatorname{det} (\rho_R r, u_R,
\phi)\,dy\,ds + \int\limits_{0}^{t} \operatorname{det}(m\ell,r,\ell_{\phi})\,ds
 + \int\limits_{0}^{t} \operatorname{det}(J_{0}r,r,r_{\phi}) \,ds  - 2 \nu\int\limits_{0}^{t}\int\limits_{\mc{F}_{0}} D(u_R) : D(\phi)\,dy\,ds \nonumber \\
& \quad
 - 2\nu\alpha\int\limits_{0}^{t}\int\limits_{\partial\mc{S}_{0}} (u_R-u_{\mc{S},R})\cdot (\phi-\phi_{\mc{S}})\ d\Gamma\,ds
 + 2\nu\alpha\int\limits_{0}^{t}\int\limits_{\partial\mc{S}_{0}} w\cdot (\phi-\phi_{\mc{S}})\,d\Gamma\,ds.
 \end{align}
 We have the following convergence results: 
 \begin{align*}
u_R &\rightharpoonup u\mbox{ in }L^2(0,T; \mc{V})\mbox{ weakly (see step $3$ of \cref{existence:bounded domain})},\\
u_R \otimes \rho_R(u_R-u_{\mc{S},R}) &\rightharpoonup u \otimes \rho(u- u_{\mc{S}})\mbox{ in }L^{\frac{4}{3}}(0,T;L^2(\mathbb{R}^3))\mbox{ weakly (see step $3$ of \cref{existence:bounded domain})},\\
  \sqrt{\rho_R}u_R &\rightarrow \sqrt{\rho} u\mbox{ strongly in }L^2(0,T;L^2_{\operatorname{loc}} (\mathbb{R}^3))\mbox{(see \cite[Proposition 1, Section 2.4]{MR1744863})},\\
 \rho_R &\rightarrow \rho\mbox{ strongly in }C([0,T];L^{q}_{\operatorname{loc}} (\mathbb{R}^3)), \ \forall\ q\in [1,\infty). 
\end{align*}
These above mentioned convergence properties allow us to pass to the limit $R\rightarrow \infty$ in each term of \eqref{weakform-momentum-bounded} and to obtain the identity \eqref{weakform-momentum}.
  \end{proof}
  \section{Discussion} \label{sec:5}
  In this section, we discuss two variants of our system and give a few remarks. Firstly, in the case of a positive initial density we mention the stronger results that can be obtained. Then, we will concentrate on the case when the fluid viscosity depends on the density. 
  
  \subsection{Positive initial density}   
We can improve the results when the initial density is away from zero $(\inf \rho_0 > 0)$, i.e., when the fluid does not contain any vacuum regions. Actually, for the approximate solution in the bounded domain, we obtain 
\begin{itemize}
\item For all $N$, $\inf \rho_N > 0$.
\item $\{u_N\}$ is bounded in $L^{\infty}(0,T; L^2(\Omega)) \cap L^2(0,T; \mc{V}_{\Omega})$.
\item Under the translation operator $\tau_h : f \mapsto f(\cdot +h)$, we can obtain from \cite[Chapter VI, Lemma VI.2.5]{MR2986590} that 
\begin{equation*}
\|\tau_h u_N-u_N\|_{L^2((0,T-h),L^2(\Omega))} \leq Ch^{1/4},
\end{equation*}
i.e., $\{u_N\}$ is bounded in $N^{1/4,2}(0,T;L^2(\Omega))$.
\item As $\mc{V}_{\Omega} \hookrightarrow L^2(\Omega)$ is compact and $u_N \in L^2(0,T;\mc{V}_{\Omega})\cap N^{1/4,2}(0,T;L^2(\Omega))$, we have that, according to \cite[Lemma 4(iv)]{MR1062395}, $u_N$ is relatively compact in $L^2(0,T;L^2(\Omega))$. This allows us to achieve the strong convergence of $u_N$ to $u$ in $L^2(0,T;L^2(\Omega))$ which is the same as in the case of a homogeneous fluid (i.e., when the fluid has constant density). 
\end{itemize}   
 Thus, in this case it is easier to justify the passage of the approximate problem \eqref{approx:integral equation} and to obtain identity \eqref{weakform-momentum}. Precisely, we obtain the following results:
 \begin{itemize}
 \item As $\inf \rho_0 > 0$, we obtain $1/ \rho \in L^{\infty}((0,T)\times\mathbb{R}^3)$. By using $u=\rho u \frac1\rho$, we have $u \in L^{\infty}(0,T; L^2(\mathbb{R}^3))$.
 \item As in \cite[Proposition 8(ii)]{MR1062395}, $u \in N^{1/4,2}(0,T;L^2(\Omega))$.
 \end{itemize}
 
  \subsection{Fluid viscosity depends on the density}
We discuss how to deal with the case if the fluid viscosity depends on density. Here we consider the fluid viscosity $\nu$ as a $C^1$ function of the fluid density and it satisfies the following:
\begin{equation}\label{hypo:viscosity}
\mbox{ there exists }\nu_1,\nu_2 >0\mbox{ such that } \nu_1 \leq  \nu(\eta) \leq \nu_2 \mbox{ for all } \eta \in \mathbb{R} \mbox{ and }\nu'\mbox{ is bounded}.
\end{equation}
In this case, we have the following relation analogous to \eqref{weakform-momentum}:
\begin{align}\label{weakform-momentum-varden}
 & \int\limits_{\mc{F}_{0}} \rho u(y,t)\cdot \phi(y,t)\,dy + m\ell(t)\cdot \ell_{\phi}(t)
 + J_{0}r(t)\cdot r_{\phi}(t)
  - \int\limits_{\mc{F}_{0}}q_0\cdot \phi(y,0)\,dy - m\ell(0)\cdot \ell_{\phi}(0) - J_{0}r(0)\cdot r_{\phi}(0)  \nonumber \\
& =  \int\limits_{\mc{F}_{0}}\int\limits_{0}^{t} \rho u \cdot \frac{\partial \phi}{\partial s}\,ds\,dy + \int\limits_{0}^{t}m\ell\cdot \ell'_{\phi}\,ds
+ \int\limits_{0}^{t}J_{0}r\cdot r'_{\phi}\,ds +\int\limits_{0}^{t}\int\limits_{\mc{F}_{0}} \left[u \otimes \rho(u-u_{\mc{S}})\right] : \nabla
\phi\,dy\,ds  \nonumber \\ &\quad - \int\limits_{0}^{t}\int\limits_{\mc{F}_{0}} \operatorname{det} (\rho r, u, \phi)\,dy\,ds  + \int\limits_{0}^{t}
\operatorname{det}(m\ell,r,\ell_{\phi})\,ds
 + \int\limits_{0}^{t} \operatorname{det}(J_{0}r,r,r_{\phi}) \,ds
- 2 \int\limits_{0}^{t}\int\limits_{\mc{F}_{0}}\nu(\rho) D(u) : D(\phi)\,dy\,ds \nonumber \\
 & \quad 
 - 2\alpha\int\limits_{0}^{t}\int\limits_{\partial\mc{S}_{0}}\nu(\rho) (u-u_{\mc{S}})\cdot (\phi-\phi_{\mc{S}})\ d\Gamma\,ds + 2\alpha\int\limits_{0}^{t}\int\limits_{\partial\mc{S}_{0}}\nu(\rho)\ w\cdot (\phi-\phi_{\mc{S}})\,d\Gamma\,ds.
 \end{align}
We can also define weak solutions in this case as previously:
 \begin{defin}
Let $T> 0$. A pair $(\rho,u)$ is a weak solution to system \eqref{chg of var mass:fluid}--\eqref{chg of var initial cond} with $\nu=\nu(\rho)$ if the following conditions hold true:
\begin{itemize} \item $\rho \geq 0, \quad \rho \in L^{\infty}((0,T)\times \mathbb{R}^3)$.  
\item $u \in L^2(0,T; \mc{V}),\quad \rho|u|^2 \in L^{\infty}(0,T; L^1(\mathbb{R}^3))$.  
\item The equation of continuity \eqref{chg of var mass:fluid} is satisfied in the weak sense. Also, a renormalized continuity equation holds in a weak sense.
 \item Balance of linear momentum holds in a weak sense, i.e., for all  $\phi \in C^{\infty}([0,T]; \mc{H})$ such that
    $\phi|_{\overline{\mc{F}_{0}}} \in C^{\infty}([0,T]\times \overline{\mc{F}_{0}})$ and for all $t \in [0,T]$, the relation \eqref{weakform-momentum-varden} holds. 
\end{itemize}
  \end{defin}
  Now we state the existence result of a weak solution, when viscosity is a function of fluid density:
  \begin{theorem}\label{main theorem-varden}
 Let $\mc{S}_0$ be a bounded, closed, simply connected set with smooth boundary and $\mc{F}_0= \mathbb{R}^3\setminus \mc{S}_0$. Assume that the self-propelled motion $w$ satisfies \eqref{self-propel normal}--\eqref{self-propel tangent}, the fluid viscosity satisfies \eqref{hypo:viscosity} and that there exist
 constants $c_1, c_2 > 0$ such that 
 \begin{align*}
&\rho_0 \geq 0,\, \rho_0 |_{\mc{S}_0} \in [c_1, c_2],\, \rho_0 \in L^{\infty}(\mathbb{R}^3),\,  u_{0} \in \mc{H},\\ & q_0=0\mbox{ a.e.\  on }\{\rho_0 = 0\},\ \frac{q_0^2}{\rho_0} \in L^1(\mc{F}_0).
 \end{align*}
 Then for an arbitrary $T>0$, there exists a weak solution $(\rho,u)$ to system \eqref{chg of var mass:fluid}--\eqref{chg of var initial
 cond} with $\nu=\nu(\rho)$.
  Moreover, we have 
 \begin{equation*}
 \inf_{\mathbb{R}^3}\rho_0 \leq \rho \leq \sup_{\mathbb{R}^3}\rho_0, \quad \rho \in C([0,T]; L^q_{\operatorname{loc}} (\mathbb{R}^3))\,\forall\ q\in [1,\infty),\quad p \in W^{-1,\infty}(0,T; L^2(\mc{F}_0)),
 \end{equation*}
and for a.e.\
  $t \in [0,T]$, the energy inequality 
    \begin{align*} 
 &\int\limits_{\mc{F}_{0}} \frac{1}{2}\rho |u|^2\,dy + \frac{m}{2}|\ell|^2
 + \frac{J_{0}}{2}|r|^2 + 2\nu_1  \int\limits_{0}^{t}\int\limits_{\mc{F}_{0}}  |D(u)|^2\,dy\,ds +
 \alpha\nu_1\int\limits_{0}^{t}\int\limits_{\partial\mc{S}_{0}} |u-u_{\mc{S}}|^2\,d\Gamma\,ds  \\
& \leq \int\limits_{\mc{F}_{0}}\frac{1}{2}\rho_0 |u_0|^2\,dy + \frac{m}{2}|\ell_0|^2 + \frac{J_{0}}{2}|r_0|^2 +
\alpha\nu_2 \int\limits_{0}^{t}\int\limits_{\partial\mc{S}_{0}}|w|^2\,d\Gamma\,ds.
 \end{align*}
  holds.
  \end{theorem}
  The proof of this theorem is similar as before.  We start with an approximation for the system on a bounded domain; later we pass to the unbounded domain. Additionally to before, we have to justify the passing of the limits in the terms: 
  \begin{align}\label{dif-term}
  & -2\int\limits_{0}^{t}\int\limits_{\mc{F}_{0}}\nu(\rho_N) D(u_N) : D(\phi)\,dy\,ds
 - 2\alpha\int\limits_{0}^{t}\int\limits_{\partial\mc{S}_{0}}\nu(\rho_N) (u_N-u_{\mc{S},N})\cdot (\phi-\phi_{\mc{S}})\ d\Gamma\,ds  \nonumber \\
 & + 2\alpha\int\limits_{0}^{t}\int\limits_{\partial\mc{S}_{0}}\nu(\rho_N)\ w\cdot (\phi-\phi_{\mc{S}})\,d\Gamma\,ds.
 \end{align}
 Now to do this, observe that we already have 
 \begin{equation*}
 \rho_N \rightarrow \rho\mbox{ strongly in }C([0,T];L^{q}(\Omega)) \quad \forall\ q\in [1,\infty).
 \end{equation*}
 With the help of hypothesis \eqref{hypo:viscosity} for the fluid viscosity, we have for all $q\in [1,\infty)$:
 \begin{equation*}
 \|\nu(\rho_N)-\nu(\rho)\|_{L^{\infty}(0,T;L^q(\Omega)} \leq \|\nu'\|_{\infty}\|\rho_N-\rho\|_{L^{\infty}(0,T;L^q(\Omega)}.
 \end{equation*}
 Thus, we have strong convergence of the viscosity 
 \begin{equation*}
 \nu(\rho_N) \rightarrow \nu(\rho) \mbox{ strongly in }L^{\infty}(0,T;L^q(\Omega) \quad \forall\ q\in [1,\infty).
 \end{equation*}
 The above strong convergence of viscosity and $$u_N \rightharpoonup u\mbox{ in }L^2(0,T; \mc{V}_{\Omega})\mbox{ weakly}$$ allow us to pass to the limit in the terms of \eqref{dif-term} $N\rightarrow \infty$. Similarly, we can proceed with the case for unbounded domain.

\subsection{Further remarks and open problems}

\begin{remark}
 Silvestre \cite{MR1953783} proved the global existence of weak solution of a fluid-structure system in the case of Dirichlet boundary conditions. She applied a global transformation and extended the self-propelled motion from $\partial \mc{S}_0$ to the whole domain. It was not possible for us to adopt her approach to our setting with Navier slip boundary conditions in order to show the corresponding result. The reasons are that we could not recover the Navier-slip boundary condition and that the extension of the self-propelled motion $w$   to the whole domain requires initial data for $w$. 
\end{remark}

\begin{remark} \label{rem:normal}In our work we require the normal component of the self-propelled velocity $W\cdot N$ and $w\cdot n$, respectively, to be zero at the boundary of the rigid body. Actually, if $w\cdot n \neq 0$, then the interface conditions \eqref{chg of var boundary-1}--\eqref{chg of var boundary-2} are replaced by 
\begin{align*} u\cdot n &= (u_{\mc{S}}+w) \cdot n \mbox{ for }y \in \partial \mc{S}_{0}, \\ (D(u)n)\times n &= -\alpha (u-u_{\mc{S}} - w)\times n \mbox{ for }y \in \partial \mc{S}_{0}. \end{align*}
In that case, instead of energy inequality \eqref{energy inequality}, we obtain
\begin{align*}
 & \int\limits_{\mc{F}_{0}} \frac{1}{2}\rho |u|^2\,dy + \frac{m}{2}|\ell|^2
 + \frac{J_{0}}{2}|r|^2 + 2 \nu\int\limits_{0}^{t}\int\limits_{\mc{F}_{0}} |D(u)|^2\,dy\,ds + \nu\alpha\int\limits_{0}^{t}\int\limits_{\partial\mc{S}_{0}} |u-u_{\mc{S}}|^2\,d\Gamma \, ds \nonumber \\
& \leq \int\limits_{\mc{F}_{0}}\frac{1}{2}\rho_0 |u_0|^2\,dy + \frac{m}{2}|\ell_0|^2 + \frac{J_{0}}{2}|r_0|^2 + \nu\alpha\int\limits_{0}^{t}\int\limits_{\partial\mc{S}_{0}}|w|^2\,d\Gamma\,ds+ \frac{1}{2}\int\limits_{0}^{t}\int\limits_{\partial \mc{S}_{0}} \rho |u|^2(w\cdot n) \,d\Gamma \, ds.
 \end{align*} 
 Here, we do not know how to bound the right hand side of the above inequality by given quantities (initial conditions and self-propelled force) as it involves the unknowns $\rho$ and $u$. That is why we impose the condition $w\cdot n=0$.
\end{remark}

\begin{remark}
As in all the works mentioned above, the elastic setting is missing. That is, it would be desirable to consider a system with a body which may be elastic and not just rigid. The difficulty of that setting is that a transformation to a fixed domain is not possible since the body may change its shape as time evolves. Hence different methods need to be developed.
\end{remark}

\section{Appendix}\label{appendix}

{\bf Derivation of the weak formulation \eqref{weakform-momentum}:} Let $\phi\in C^{\infty}([0,T]; \mc{H})$ such that $\phi|_{\overline{\mc{F}}_{0}} \in C^{\infty}([0,T] \times \overline{\mc{F}}_{0})$. We show that \eqref{weakform-momentum} is the weak form of our system \eqref{chg of var mass:fluid}--\eqref{chg of var initial cond}. To this end we multiply the momentum equation \eqref{chg of var momentum:fluid} formally by $\phi$ and integrate over $\mc{F}_{0} \times (0,t)$, which yields
 \begin{equation}\label{modified-momentum}
\int\limits_{0}^{t}\int\limits_{\mc{F}_{0}}\left[\frac{\partial }{\partial s}(\rho u) + \operatorname{div}\left[u \otimes \rho(u-u_{\mc{S}})\right]+
\rho r \times u + \nabla p - \operatorname{div}(2\nu D(u)) \right]\cdot \phi\,dy\,ds= 0. \end{equation} 
In the following we will consider each term of this identity separately; in particular we will integrate by parts several times and will apply the other equations of our system.\\
{\it{First term}}. \begin{align}\label{chg of var momentum:fluid-term1} \int\limits_{\mc{F}_{0}}\int\limits_{0}^{t} 
\frac{\partial }{\partial s}(\rho u) \cdot \phi\,ds\,dy = 
-\int\limits_{\mc{F}_{0}}\int\limits_{0}^{t} \rho u \cdot \frac{\partial \phi}{\partial s}\,ds\,dy  + 
\int\limits_{\mc{F}_{0}} \Big( \rho u(y,t)\cdot \phi(y,t) - q_0\cdot \phi(y,0)\Big)\,dy .  \end{align} \vspace{0.2cm}\\ 
{\it{Second term}}. \begin{equation*} \int\limits_{0}^{t}\int\limits_{\mc{F}_{0}} \operatorname{div}\left[u \otimes \rho(u-u_{\mc{S}})\right]\cdot \phi\,dy\,ds= - \int\limits_{0}^{t}\int\limits_{\mc{F}_{0}} \left[u \otimes \rho(u-u_{\mc{S}})\right] : \nabla \phi \,dy\,ds+ \int\limits_{0}^{t}\int\limits_{\partial\mc{S}_{0}} \rho((u-u_{\mc{S}}) \cdot n)(u \cdot \phi)\,d\Gamma\,ds, \end{equation*} where for the boundary term, we have used the relation
 \begin{equation*} (u \otimes v)w= (v\cdot w)u, \quad \mbox{for}\quad u,v,w \in \mathbb{R}^3. \end{equation*} 
As $u\cdot n = u_{\mc{S}} \cdot n \mbox{ for }y \in \partial \mc{S}_{0}$, 
we obtain 
\begin{equation*} 
\int\limits_{0}^{t}\int\limits_{\mc{F}_{0}} \operatorname{div}\left[u \otimes \rho(u-u_{\mc{S}})\right]\cdot \phi\,dy\,ds= - \int\limits_{0}^{t}\int\limits_{\mc{F}_{0}} \left[u \otimes \rho(u-u_{\mc{S}})\right] : \nabla \phi \,dy\,ds. \end{equation*} \vspace{0.2cm}\\ 
{\it{Third term}}. \begin{equation*} 
\int\limits_{0}^{t}\int\limits_{\mc{F}_{0}} (\rho r \times u) \cdot \phi\,dy\,ds = \int\limits_{0}^{t}\int\limits_{\mc{F}_{0}} \operatorname{det} (\rho r, u, \phi)\,dy\,ds. \end{equation*} \vspace{0.2cm}\\ 
{\it{Fourth term}}. \begin{equation*}
\int\limits_{0}^{t}\int\limits_{\mc{F}_{0}} \nabla p \cdot \phi\,dy\,ds = \int\limits_{0}^{t}\int\limits_{\partial \mc{S}_{0}} pn\cdot \phi\,d\Gamma\,ds = \int\limits_{0}^{t}\int\limits_{\partial \mc{S}_{0}} pn\cdot \phi_{\mc{S}}\,d\Gamma\,ds = \int\limits_{0}^{t}\int\limits_{\partial \mc{S}_{0}} pn\cdot (\ell_{\phi} + r_{\phi} \times y)\,d\Gamma\,ds, \end{equation*} 
with $\phi_{\mc{S}}$ as in \eqref{eqn:phiS}.  \vspace{0.2cm}\\
{\it{Fifth term}}.
 We analyze this term as in \cite[Lemma 1]{MR3165279} 
and obtain 
\begin{align}\label{chg of var momentum:fluid-term5}  &-\int\limits_{0}^{t}\int\limits_{\mc{F}_{0}} \operatorname{div}(2\nu D(u)) \cdot \phi\,dy\,ds \nonumber \\ &= 2 \nu\int\limits_{0}^{t}\int\limits_{\mc{F}_{0}} D(u) : D(\phi)\,dy\,ds - 2\nu \int\limits_{0}^{t}\int\limits_{\partial\mc{S}_{0}} D(u)n \cdot \ell_{\phi}\,d\Gamma\,ds \nonumber \\ &\quad  - 2\nu \int\limits_{0}^{t}\int\limits_{\partial\mc{S}_{0}} ( y \times D(u)n) \cdot r_{\phi}\,d\Gamma\,ds  - 2\nu \int\limits_{0}^{t}\int\limits_{\partial\mc{S}_{0}} (D(u)n \times n)\cdot [(\phi-\phi_{\mc{S}}) \times n]\,d\Gamma\,ds \end{align} 
After plugging the terms \eqref{chg of var momentum:fluid-term1}--\eqref{chg of var momentum:fluid-term5} into the relation \eqref{modified-momentum}, we get 
\begin{align*} &-\int\limits_{0}^{t} \int\limits_{\mc{F}_{0}}\rho u \cdot \frac{\partial \phi}{\partial s}\,dy\,ds + \int\limits_{\mc{F}_{0}} \Big(\rho u(y,t)\cdot \phi(y,t) - q_0\cdot \phi(y,0)\Big)\,dy   - \int\limits_{0}^{t}\int\limits_{\mc{F}_{0}} \left[u \otimes \rho(u-u_{\mc{S}})\right] : \nabla \phi\,dy\,ds\\
&\quad  + \int\limits_{0}^{t}\int\limits_{\mc{F}_{0}} \operatorname{det} (\rho r, u, \phi)\,dy\,ds
+ 2 \nu\int\limits_{0}^{t}\int\limits_{\mc{F}_{0}} D(u) : D(\phi)\,dy\,ds \\ &= 2\nu \int\limits_{0}^{t}\int\limits_{\partial\mc{S}_{0}} (D(u)n \times
n)\cdot [(\phi-\phi_{\mc{S}}) \times n]\,d\Gamma\,ds  -\int\limits_{0}^{t}\int\limits_{\partial \mc{S}_{0}} pn\cdot (\ell_{\phi} + r_{\phi} \times
y)\,d\Gamma\,ds \\
 &\quad + 2\nu \int\limits_{0}^{t}\int\limits_{\partial\mc{S}_{0}} D(u)n \cdot \ell_{\phi}\,d\Gamma\,ds + 2\nu
 \int\limits_{0}^{t}\int\limits_{\partial\mc{S}_{0}} ( y \times D(u)n) \cdot r_{\phi}\,d\Gamma\,ds .
\end{align*} 
We use equations \eqref{chg of var linear momentum:body}--\eqref{chg of var angular momentum:body} to rewrite the preceding
relation as \begin{align}\label{bdary replace} & -\int\limits_{0}^{t} \int\limits_{\mc{F}_{0}} \rho u \cdot \frac{\partial \phi}{\partial s}\,dy \,ds +
 \int\limits_{\mc{F}_{0}} \Big(\rho u(y,t)\cdot \phi(y,t) - q_0\cdot \phi(y,0)\Big)\,dy   - \int\limits_{0}^{t}\int\limits_{\mc{F}_{0}} \left[u \otimes
\rho(u-u_{\mc{S}})\right] : \nabla \phi\,dy\,ds \nonumber\\ & \quad+ \int\limits_{0}^{t}\int\limits_{\mc{F}_{0}} \operatorname{det} (\rho r, u, \phi)\,dy\,ds + 2
\nu\int\limits_{0}^{t}\int\limits_{\mc{F}_{0}} D(u) : D(\phi)\,dy\,ds \nonumber\\
& = 2\nu \int\limits_{0}^{t}\int\limits_{\partial\mc{S}_{0}} (D(u)n \times n)\cdot
[(\phi-\phi_{\mc{S}}) \times n]\,d\Gamma\,ds \nonumber \\
& \quad - \int\limits_{0}^{t}m\ell'\cdot \ell_{\phi}\,ds + \int\limits_{0}^{t}(m\ell \times r)\cdot \ell_{\phi}\,ds - \int\limits_{0}^{t}J_{0}r'\cdot
 r_{\phi}\,ds + \int\limits_{0}^{t}(J_{0}r \times r)\cdot r_{\phi}\,ds .
\end{align}
Now, integration by parts with respect to time yields
  \begin{align*} 
   - \int\limits_{0}^{t}m\ell'\cdot \ell_{\phi}\,ds - \int\limits_{0}^{t}J_{0}r'\cdot r_{\phi}\,ds 
  & = \int\limits_{0}^{t}m\ell\cdot \ell'_{\phi}\,ds -
  m\ell(t)\cdot \ell_{\phi}(t)\,ds
  + m\ell(0)\cdot \ell_{\phi}(0) \nonumber\\
  & \quad + \int\limits_{0}^{t}J_{0}r\cdot r'_{\phi}\,ds
  -J_{0}r(t)\cdot r_{\phi}(t) + J_{0}r(0)\cdot r_{\phi}(0).
  \end{align*}
Regarding the first term on the right hand side of \eqref{bdary replace}, we observe that by \eqref{chg of var boundary-2} 
 \begin{align*}
 \int\limits_{\partial\mc{S}_{0}} (D(u)n \times n)\cdot
[(\phi-\phi_{\mc{S}}) \times n]\,d\Gamma
 = 
-\alpha\int\limits_{\partial\mc{S}_{0}} [(u-u_{\mc{S}}-w) \times n]\cdot [(\phi-\phi_{\mc{S}}) \times n] \ d\Gamma.
 \end{align*}
The identity $(A\times B)\cdot(C\times D)=(A\cdot C)(B\cdot D)-(A\cdot D)(B\cdot C)$ for any $A,B,C,D \in \mathbb{R}^3$ together with $(u-u_{\mc{S}}-w) \cdot n =0$ by \eqref{self-propel normal} and \eqref{chg of var boundary-1} yield
 \begin{align} \label{eqn:oldprop}
 \int\limits_{\partial\mc{S}_{0}} (D(u)n \times n)\cdot
[(\phi-\phi_{\mc{S}}) \times n]\,d\Gamma
 = -\alpha\int\limits_{\partial\mc{S}_{0}} (u-u_{\mc{S}} -w)\cdot (\phi-\phi_{\mc{S}}) d\Gamma.
 \end{align}
Putting \eqref{bdary replace}--\eqref{eqn:oldprop} together, we obtain the weak form  \eqref{weakform-momentum} of our system.


\section*{Acknowledgements}
 The work of \v S. Ne\v casov\'a was supported by the Czech Science Foundation grant GA19-04243S in the framework of RVO 67985840. This work was initiated while \v S.N.\ and A.R. were visiting W\"urzburg, \v S.N. as a Giovanni-Prodi visiting professor.
\bibliography{reference} \bibliographystyle{siam}

\end{document}